\documentclass[twoside]{article}
\usepackage{imsart}
\usepackage{amsfonts,amsmath,latexsym,amsthm,amssymb}
\usepackage{dsfont}     
\usepackage{mathtools}
\usepackage[english,francais]{babel}
\usepackage[applemac]{inputenc}    
\usepackage[T1]{fontenc}
\newtheorem{prop}{Proposition}[section]
\newtheorem{rema}[prop]{Remark}
\newtheorem{teor}[prop]{Theorem}
\newtheorem{defi}[prop]{Definition}
\newtheorem{assu}[prop]{Assumption}
\newtheorem{lemm}[prop]{Lemma}
\newtheorem{coro}[prop]{Corollary}
\newtheorem{exam}[prop]{Example}
\newtheorem{ass}[prop]{Assumption}

\def\1{\mathds{1}}
\newcommand{\filtrationf}{$\mathbb{F}$}
\newcommand{\filtrationg}{$\mathbb{G}$}
\newcommand{\R}{\mathbb{R}}
\newcommand{\beqn}{\begin{eqnarray}}
\newcommand{\eeqn}{\end{eqnarray}}
\newcommand{\nonu}{\nonumber}

\newcommand{\pqa}{\left[}
\newcommand{\pqc}{\right]}
\newcommand{\pga}{\left\{}
\newcommand{\pgc}{\right\}}
\newcommand{\pta}{\left(}
\newcommand{\ptc}{\right)}
\newcommand{\vaa}{\left|}
\newcommand{\vac}{\right|}
\newcommand{\shc}{\mathcal{C}}
\newcommand{\shl}{\mathcal{L}}
\newcommand{\shg}{\mathcal{G}}
\newcommand{\shf}{\mathcal{F}}
\newcommand{\shd}{\mathcal{D}}
\newcommand{\sha}{\mathcal{A}}
\newcommand{\shh}{\mathcal{H}}
\newcommand{\shm}{\mathcal{M}}

\newcommand{\sht}{\mathcal{T}}
\newcommand{\shv}{\mathcal{V}}
\newcommand{\shw}{\mathcal{W}}
\newcommand{\ep}{\varepsilon}
\newcommand{\lra}{\longrightarrow}

\newcommand{\interval}{0\leq t \leq 1}
\newcommand{\intervala}{0\leq t< 1}
\def\be{\begin{equation}}
\def\ee{\end{equation}}
\def\ba{\begin{array}}
\def\ea{\end{array}}
\begin{document}
\begin{frontmatter}
\title{On stochastic calculus related to financial assets without
 semimartingales}
\author[HSBC]{Rosanna COVIELLO}
\author[luiss]{, Cristina DI GIROLAMI}
\ead{cdigirolami@luiss.it}
\author[ensta,inria]{and Francesco RUSSO}
\ead{francesco.russo@ensta-paristech.fr}
\address[HSBC]{HSBC, 103, av. des Champs-Elys\'ees, F-75419 Paris Cedex 09 (France).}
\address[luiss]{LUISS Guido Carli - Libera
    Universit\`a Internazionale degli Studi Sociali Guido Carli di Roma (Italy).}
\address[ensta]{ENSTA ParisTech,
Unit\'e de Math\'ematiques appliqu\'ees,
32, Boulevard Victor,
F-75739 Paris Cedex 15 (France).}
\address[inria]{INRIA Rocquencourt 
and Cermics Ecole des Ponts, Projet MATHFI. Domaine de Voluceau, BP 105
F-78153 Le Chesnay Cedex (France).}

\maketitle
\date{February 9th, 2011}
\selectlanguage{english}

\begin{abstract} \
This paper does not suppose a priori that the evolution of 
the price of a financial asset
is a semimartingale. Since possible strategies of investors are self-financing,
previous prices are forced to be finite quadratic variation processes.
The non-arbitrage property is not excluded if the class $\sha$ of admissible strategies
is restricted. The classical notion of martingale is replaced with the notion
of $\sha$-martingale. A calculus related to  $\sha$-martingales with some examples
is developed. Some applications to no-arbitrage, viability, hedging and the maximization of the utility of an
insider are expanded. We finally revisit some no arbitrage conditions of
 Bender-Sottinen-Valkeila type.
\end{abstract}


\bigskip

\end{frontmatter}
\selectlanguage{english}

\section{Introduction}

This article is devoted to the memory of Professor Paul Malliavin,
a legend in mathematics. Among his huge and fruitful contributions
there is the celebrated Malliavin calculus. Malliavin calculus
was succesfully applied to many areas in probability and analysis but also 
in financial mathematics. Prof. Malliavin himself in the last
part of his career was very productive in this field, 
 as the excellent monograph \cite{MallThal} written with
Prof. Thalmaier shows.  Especially the third named author is
grateful for all the mathematical interactions he could have
with him. Prof. Malliavin was actively present to a talk of F. Russo 
which included the first part of the topic of the present paper.

According to 
the fundamental theorem of asset pricing of Delbaen and Schachermayer in \cite{DSbook}, Chapter 14,
 in absence of \textit{free lunches with vanishing risk} (NFLVR), when investing possibilities run only through simple predictable strategies  with respect to some filtration $\mathbb{G}$,
the price process of the risky asset $S$  is forced to be a semimartingale.
However (NFLVR) condition could not be reasonable in several situations. In that case $S$ may not be a semimartingale.
We illustrate here some of those circumstances.

Generally, admissible strategies are let vary in a quite large class of
predictable processes with respect to some filtration $\mathbb{G}$, representing the information flow available to the investor. 
As a matter of fact, the class of admissible strategies could be reduced because of  different market regulations 
or for practical reasons. For instance, the investor  could not be allowed to hold more than a certain number of stock shares.
On the other hand it could be realistic to impose a minimal delay between
two possible transactions as suggested by Cheridito (\cite{che}), see also \cite{HJP}: when the logarithmic price $\log(S)$ is a geometric fractional Brownian motion (fbm), it is impossible to realize arbitrage possibilities satisfying that minimal requirement. We remind that without that restriction,
the market admits arbitrages, see for instance \cite{rogers, sh, salopek}.
When the logarithmic price of $S$ is a geometric fbm or some particular strong Markov process, arbitrages can be excluded taking into account proportional transactions costs: Guasoni (\cite{gua}) has shown that, in that case,  the class of admissible strategies has to be restricted to bounded variation processes and this rules out arbitrages.

Besides the restriction of the class of admissible strategies, the adoption of non-semimartingale models finds its justification when the no-arbitrage condition itself is not likely. 

Empirical observations  reveal, indeed, that $S$ could fail to be a semimartingale because of market imperfections due to micro-structure noise,
as intra-day effects.
A model which considers those imperfections 
 would add to  $W$, the Brownian motion describing  $\log$-prices, a zero quadratic variation process, as a fractional Brownian motion of Hurst index greater than $ \frac{1}{2}$,
see for instance \cite{wor}.
Theoretically arbitrages in very small time interval could be possible, which would be compatible
with the lack of semimartingale property.

At the same way if (FLVR) are not possible for an \textit{honest investor}, an {\it inside trader}
could realize a free lunch with respect to the  enlarged filtration $\mathbb{G}$ 
including the one generated by prices and the extra-information.
Again in that case  $S$ may not be a semimartingale.
The literature concerning inside trading and asymmetry of information has been extensively enriched by several papers in the last ten years;
among them we quote Pikowski and Karatzas (\cite{pikar}),  Grorud and Pontier (\cite{gropon}),
 Amendinger, Imkeller and Schweizer (\cite{AIS}). They adopt enlargement of filtration techniques to describe the  evolution of stock prices in the insider filtration. 

Recently, 
some authors approached the problem in a new way using in particular forward integrals, in the framework of 
stochastic calculus via regularizations. For a comprehensive survey of that calculus  see \cite{RV05}.
Indeed, forward integrals could  exist also for
non-semimartingale integrators.
Leon, Navarro and Nualart in \cite{LNN}, for instance, solve the problem of maximization of expected logarithmic utility of an agent who holds an initial information depending on the future of prices.
They operate under  technical conditions which, a priori, do not imply the classical  Assumption (H') for enlargement considered 
in \cite{jacod}. 
Using  forward integrals, they determine the utility maximum.
However, a posteriori, they found out that  their conditions oblige 
 $S$ to be a semimartingale.

Biagini and {\O}ksendal (\cite{BO}) considered somehow the converse implication. Supposing that the maximum utility is attained, they proved that
$S$ is a semimartingale.
Ankirchner and Imkeller (\cite{AI})  continue to develop the enlargement of filtrations techniques and show, among  other thinks, a similar result as \cite{BO}
using the fundamental theorem of asset pricing  of Delbaen-Schachermayer. In particular they establish a link 
between that fundamental theorem 
and finite utility. 

In our paper we treat a market where there are one risky asset, whose price is a strictly positive process $S$,
and a \textit{less risky} asset with price $S^0$, possibly  riskless but a priori only with bounded variation. A class $\sha$ of admissible trading strategies is 
specified. If $\sha$ is not large enough to generate all predictable simple strategies, then $S$ has no need to be a semimartingale,
even requiring the absence of free lunches among those strategies.

The aim of the present paper is to settle the basis of a fundamental 
(even though preliminary) calculus 
which, in principle, allows to model financial assets without semimartingales.
 Of course this constitutes the first step of a more involved theory
generalizing the classical theory related to semimartingales.
The objective is two-fold.
\begin{enumerate}
\item To provide a mathematical framework which extends It\^o calculus
 conserving some particular aspects of it in a non-semimartingale framework.
This has an interest in itself,
independently from mathematical finance.
The two major tools are {\it forward integrals} and
$\sha$-{\it martingales}. 
\item 
To build the basis of a corresponding financial theory 
which allows to deal with several problems as hedging 
 and non-arbitrage pricing, viability and completeness 
as well as with utility maximization. 
\end{enumerate}

For the sake of simplicity in this introduction we suppose that the less risky  asset $S^0$ is constant and equal to 1.

As anticipated, a natural tool to describe the self-financing condition
is the forward integral of an integrand process $Y$ with respect to an integrator $X$,  denoted by $\int_0^t Y d^-X$; see section \ref{s2} for definitions.
Let $\mathbb{G} = (\shg_t)_{0 \le t \le 1}$ be a filtration on an underlying probability space $(\Omega, \shf, P)$, with $\shf = \shg_1$; $\mathbb{G}$
represents the flow of information available to the investor.
A {\bf self-financing portfolio} is a pair $(X_0,h)$ where $X_0$ is the initial value of the portfolio and $h$ is a  $\mathbb{G}$-adapted   and
 $S$-forward integrable process  specifying 
the number of  shares of  $S$ held in the portfolio. The market value process $X$ of such a portfolio, is given by $X_0+\int_0^\cdot h_sd^-S_s$, while 
$h^0_t = X_t - S_t h_t$ constitutes the number of shares of the less risky asset held.

 This  formulation of  self-financing condition is 
coherent with the case of transactions at fixed discrete dates.
 Indeed, let we consider a \textit{buy-and-hold strategy},
 i.e. a pair $(X_0,h)$ with $h= \eta I_{(t_0,t_1]}, 0 \le t_0 \le t_1 \le 1, $  and $\eta$ being a $\shg_{t_0}$-measurable random variable.
 Using the definition of forward integral it is not difficult to see that:
 $X_{t_0}=X_0,$  $X_{t_1}=X_0+\eta(S_{t_1}-S_{t_0})$. This implies 
  $h^0_{t_0 +}=X_0-\eta S_{t_0},$ $h^0_{t_1 +}=X_0+\eta(S_{t_1}-S_{t_0})$ and   
\beqn \label{1}
X_{t_0} = h_{{t_0} +}  S_{t_0} + h^0_{t_0 +},
\quad X_{t_1} = h_{t_1 +} S_{t_1} + h^0_{t_1 +}:
\eeqn
at the \textit{re-balancing} dates $t_0$ and $t_1,$ the value of the old portfolio must be reinvested to build the new portfolio without exogenous withdrawal of money. By $h_{t+}$, we denote ${\rm lim}_{s \downarrow t} h_s$.
The use of forward integral or other pathwise type integral is crucial.
Previously  some other functional integrals as Skorohod type integrals,
involving Wick products see for instance  \cite{biaginifbm}.
They are however not economically so appropriated as for instance
\cite{BHult} points out.
   
In this paper $\sha$ will be a real linear subspace of all  self-financing portfolios  and it will constitute, by definition, the class of all {\bf admissible} portfolios.
 $\sha$  will depend on the kind of problems one  has to face:
hedging, utility maximization,  modeling inside trading.
If we require that $S$ belongs to $\sha$, then the process $S$ is forced to be a finite quadratic variation process.
In fact, $\int_0^\cdot S d^-S$ exists if and  only if the quadratic variation $[S]$ exists, see \cite{RV05}; in particular
one would have 
$$  
\int_0^\cdot S_s d^- S_s =  \frac{1}{2} (S^2 - S^2_0 - [S]). 
$$
However, there could be situations in which $S$ may be allowed not to
have finite quadratic variation. In fact a process $h$ could be 
theoretically an integrand of a process $S$ without finite quadratic variation
if it has for instance bounded variation.

Even if the price process $(S_t)$ 
is an $(\shf_t)$-adapted process, the class $\sha$ 
is first of all a class of integrands of $S$.
We recall  the significant result of \cite{RV} Proposition 1.2.
Whenever $\sha$ includes the class of bounded $(\shf_t)$-processes then $S$
is forced to be a semimartingale.
In general, the class of forward  integrands
with respect to $S$
 could be much different from 
 the set
of locally bounded predictable $(\shf_t)$-processes. 

A crucial concept is provided by
 $\sha$-martingale processes. Those processes naturally intervene in utility maximization, arbitrage and 
uniqueness of hedging prices.

A process $M$  is said to be an  $\sha$-{\bf martingale}
if for any process $Y \in \sha$, 
$$
\mathbb{E} \pqa \int_0^\cdot Y d^-M \pqc= 0.
$$
If for some filtration   $\mathbb{F}$ with respect to which $M$ is adapted, 
   $\sha$ contains the class of all bounded $\mathbb{F}$-predictable processes, then $M$ is an $\mathbb{F} $-martingale.

$\shl$ will be the sub-linear space of $L^0(\Omega)$ 
representing a  set of \textbf{contingent claims of interest} for one investor.
An $\sha$-{\bf attainable contingent claim} will be  a random variable $C$ for which  there is a self-financing portfolio  $(X_0, h)$ with $h \in \sha$ and 
$$ C = X_0 + \int_0^1 h_s d^- S_s. $$ $X_0$ will be called {\bf replication price} for $C$.


A portfolio $(X_0,h)$ is said to be an $\sha$-arbitrage if $h\in \sha$, $X_1 \ge X_0 $ almost surely and
$ P\{X_1 - X_0 > 0\} > 0$. 
We denote by $\shm$ the set of probability measures being equivalent to the initial probability  $P$  under which 
$S$ is an $\sha$-martingale.
If $\shm$ is non empty then the market is $\sha$-{\bf arbitrage free}. 
In fact if  $Q \in \shm$, 
 given a pair $(X_0,h)$ which is an $\sha$-arbitrage, then $E^Q [X_1 - X_0] = E^Q [\int_0^1 h d^-S] = 0$.
In that case the replication price $X_0$ of an $\sha$-attainable contingent claim $C$ is unique,
 provided that the process $h\eta,$
for any bounded random variable $\eta$ in $\shg_0$ and $h$ in $\cal A,$ still  belongs to $\sha$.
Moreover  $X_0 = E^Q [C \vert \shg_0]$.
In reality, under the weaker assumption that  the market is $\sha$-arbitrage free,  the replication price is still unique, see Proposition \ref{p5.33}.
   
Using the inspiration coming from \cite{BSV}, we reformulate 
a non-arbitrage property related to an underlying $S$
which verifies the so-called full support condition.
We provide some theorems which generalize some aspects
of \cite{BSV}. 
Let us denote by $S_s(\cdot)$ the history 
at time $s$ of process $S$. 
If  $S$ is has
 finite quadratic variation and
$[S]_t = \int_0^t \sigma^2(s, S_s(\cdot)) S_s^2 ds, t \in [0,T]$
and $\sigma:[0,T] \times C([-1,0] \rightarrow \R$ is continuous, bounded
and non-degenerate.
It is possible to provide rich classes $\sha$ of strategies
excluding arbitrage opportunities. See 
 Proposition \ref{propAssA}, Example \ref{EVTeta} 
and the central Theorem \ref{teoNOAAR}. However,
since there are many examples of non-semimartingale
processes $S$ fulfilling the full support conditions,
the class $\sha$ will not generate the 
canonical filtration of $S$. 

The market will be said $(\mathcal{A},\shl)$-{\bf attainable} if 
every element of $\shl$ is $\sha$-attainable.
If the market is $(\mathcal{A},\shl)$-attainable
 then all the probabilities measures in $\shm$ coincide on $\sigma(\shl)$, 
see Proposition \ref{p5.30}.
If  $\sigma(\shl) = \shf$ then $\shm$ is a singleton: this result
recovers the classical case, i.e. there is a unique probability measure
under which $S$ is a semimartingale.


In these introductory lines we will focus  on
one particular toy model.

For simplicity we illustrate the case where $[\log (S)]_t=\sigma^2 t, \quad \sigma > 0. $    
We choose as $\shl$ the set of all \textit{European} contingent claims $C = \psi(S_1) $ where $\psi$ is continuous with polynomial growth. 
We consider the case $\sha=\mathcal{A}_S,$ where 
\beqn \nonu
\mathcal{A}_S&=&\pga (u(t,S_t)), \intervala
\left |\right. u:[0,1]\times \mathbb{R} \rightarrow \mathbb{R}, \mbox{ Borel-measurable }\right. \\ \nonu &&\left.
\mbox{ with polynomial growth}  \pgc.
\eeqn 
If the corresponding $\mathcal{M}$ is non empty and $\sha = \sha_S,$ as 
assumed in this section, the law of $S_t$ has to be equivalent to Lebesgue measure 
for every  $0 < t  \le 1$, 
see Proposition  \ref{p5.27}.

An example of $\sha$-martingale is the so called
 {\bf weak Brownian motion of order} $k = 1$ and quadratic variation equal to $t$.
That notion was introduced in \cite{FWY}: a weak Brownian motion of
order $1$ is a process $X$ such that the law of $X_t$ is $N(0,t)$ for
any $t \ge 0$.

Such a market is $(\mathcal{A},\shl)$-attainable: in fact, a random variable  $C = \psi(S_1)$ is an $\sha$-attainable contingent claim. To build a replicating strategy the investor has to 
choose $v$ as solution of the following problem 
$$ \left \{
\begin{array}{lll}
\partial_t v(t,x) + \frac{1}{2}\sigma^2 x^2  \partial^{(2)}_{xx} v(t,x) &= & 0 \\
v(1,x) &=& \psi(x) 
\end{array}
\right. $$
and $X_0 = v(0,S_0).$
This follows easily after application of It\^o's formula contained in 
Proposition \ref{p2.1}, see Proposition \ref{p5.31}.
This technique was introduced by \cite{KS}.
Subsequent papers in that direction are those of \cite{Z} and
 \cite{BSV} which shows in particular that several path dependent options
can be covered only assuming that $S$ has the same quadratic variation as 
geometrical Brownian motion. 
In Proposition \ref{PAsian} and in Proposition \ref{EDPtratti},
we highlight in particular that this
 method can be adjusted to
 hedge also \textit{Asian} contingent claims and some options only 
depending on a finite number of dates of the underlying price.
This discussion is continued in \cite{DGR} and \cite{DGRnote} which 
perform  a suitable infinite dimensional  calculus via regularizations, 
opening the way to the possible hedge of much
reacher classes of path dependent options.




Given an utility function satisfying usual assumptions, it is possible to show that the
maximum $\pi$ is attained on a class of portfolios  fulfilling
conditions related to Assumption \ref{a5.40},
if and only if there exists a probability measure under which $\log(S)-\int_0^\cdot \pta \sigma^2 \pi_t-\frac{1}{2}\sigma^2\ptc dt$ is an $\cal{A}$-martingale, see Proposition \ref{p5.45}. Therefore if $\sha$ is big enough to fulfill
conditions related to Assumption $\shd$ in Definition \ref{d3.6}, 
then $S$ is a classical semimartingale.

Before concluding we introduce some examples of motivating
pertinent classes $\sha$.
\begin{enumerate}
\item {\bf Transactions at fixed dates}.
Let $0 = t_0 < t_1 < \ldots < t_m  = 1$ be 
a fixed subdivision of $[0,1]$.
The price process is continuous but the transactions
take place at the fixed considered dates.
$\sha$ includes the class of predictable processes of the type
$$ H_t = \sum_{i=0}^{n-1} H_{t_i} 1_{(t_i, t_{i+1}]},$$
where $H_{t_i}$ is an $\shf_{t_i}$ measurable random variable.
A process $S$ such that $(S_{t_i})$ is an $(\shf_{t_i})$
-martingale is an $\sha$-martingale.
\item{\bf Cheridito type strategies}.
According to  \cite{che, HJP}, that class $\sha$ of strategies
 includes bounded processes $H$ such that
the time between two transactions is greater or equal than $\tau$
for some $\tau >0$.
\item{\bf Delay or anticipation}.
If $\tau \in \mathbb{R} $,
 then  $\sha$ is constituted by integrable processes $H$
such that $H_{t}$ is   $\shf_{(t + \tau)_+}$-measurable.
A process $S$ which is an $(\shf_{{(t+\tau})_+})$ martingale is an
$\sha$-martingale.
\item Let $G$ be an anticipating random variable with respect to 
$\mathbb{F}$. $\sha$ is a class of processes
of the type
$H_t = h(t,G)$, $h(t,x)$ is a random field fulfilling some
Kolmogorov continuity lemma in $x$.
\item Other examples are described in section 4.
\end{enumerate}

Those considerations show that
most of the classical results of basic financial 
 theory admit a natural extension to
 non-semimartingale models.

The paper is organized as follows. 
In section 2, we introduce stochastic calculus 
via regularizations for forward integrals.
Section 3 considers, a priori, a class $\sha$ of integrands associated
with some integrator $X$ and focuses the notion of
 $\sha$-martingale with respect to $\sha$.
We explore the relation between $\sha$-martingales and weak Brownian 
motion; later we discuss
the link between the existence of a maximum for a an optimization
problem and the $\sha$-martingale property. 

The class $\sha$ is related to classes of 
{\bf admissible} strategies of an investor.
The admissibility   concern
    theoretical or financial (regulatory) restrictions.
At the theoretical level,  classes of admissible strategies
 are introduced using Malliavin
calculus, substitution formulae and It\^o fields.
 Regarding finance applications, the class of strategies defined
 using Malliavin calculus 
is useful when   $\log(S)$ is a geometric Brownian motion with respect
 to a filtration $\mathbb{F}$ contained
in $\mathbb{G}$; the use of substitution formulae naturally appear
 when trading 
with an initial extra information, already available at time $0$; 
 It\^o fields apply whenever $S$ is a generic finite quadratic 
variation process.
Section 4 discusses some of previous examples and it
 deals with basic applications to mathematical finance. 
We define self-financing portfolio strategies and we provide examples.
Technical problems related to the use of forward integral in order to 
describe the evolution of the wealth process appear.
 Those problems arise because of the lack of
chain rule properties.
Later, we discuss absence of $\sha$-arbitrages,
$(\mathcal{A},\shl)$-attainability and hedging.
In Section 5 we
analyze the problem of maximizing expected utility from terminal wealth.
 We obtain results about the existence of an \textit{optimal portfolio} 
generalizing those of \cite{LNN} and \cite{BO}.

\section{Preliminaries}\label{s2}
For the convenience of the reader we give some basic concepts and
fundamental results about stochastic calculus with respect to finite
quadratic variation processes which will be extensively used later.
For more details we refer the reader to \cite{RV05}.

In the whole paper $\pta \Omega,
\mathcal{F},P\ptc$ will be a fixed probability space.
For a stochastic process $X=(X_t,\interval)$  defined on $\pta \Omega,
\mathcal{F},P\ptc$ we will adopt the convention $X_t=X_{(t\vee
  0)\wedge 1},$ for $t$ in $\mathbb{R}.$
Let $0\leq T\leq 1.$ 
We will say that a sequence of processes $\pta X^n_t, 0\leq t\leq T \ptc_{n\in_{\mathbb{N}}}$ \textbf{converges uniformly in probability (ucp) on $[0,T]$} toward a process $(X_t, 0\leq t\leq T),$ if $\sup_{t\in[0,T]}\vaa X^n_t-X_t\vac$ converges to zero in probability. 

\begin{defi}
\begin{enumerate}
\item
Let $X=(X_t,0\leq t\leq T)$ and $Y=(Y_t,0\leq t\le T)$ be processes with paths respectively in $C^0([0,T])$ and $L^1([0,T])$. Set, for every $0\leq t \leq T$,
$$
I(\ep,Y,X,t)=\frac{1}{\ep}\int_0^t Y_s \pta X_{s+\ep}-X_s \ptc ds, 
$$
and 
$$
C(\ep,X,Y,t)= \frac{1}{\ep} \int_0^t\pta Y_{s+\ep}-Y_s\ptc \pta X_{s+\ep}-X_s\ptc ds.
$$
If $I(\ep,Y,X,t)$ converges in probability for every $t$ in $[0,T],$ and the limiting process admits a continuous version $I(Y,X,t)$ on $[0,T],$
${Y}$ is said to be $\textbf{X}$\textbf{-forward integrable} on $[0,T]$. The process $\pta I(Y,X,t), 0\leq t\leq T\ptc$ is denoted by $\int_0^\cdot Yd^-X.$ 
If $I(\ep,Y,X,\cdot)$ converges $ucp$ on $[0,T]$ we will say that the forward integral $\int_0^\cdot Yd^-X$ is the \textbf{limit ucp of its regularizations}.
\item 
If $(C(\ep,X,Y,t), 0\leq t\leq T)$ converges ucp on $[0,T]$ when $\ep$ tends to zero, the limit will be called the \textbf{covariation process} between $X$ and $Y$   and it will be denoted by  $[X,Y].$
If $X=Y,$ $\pqa X,X\pqc$ is called the \textbf{finite quadratic variation} of $X$: it will also be denoted by $\pqa X\pqc,$ and $X$ will be said to be a \textbf{finite quadratic variation process} on $[0,T].$ 
\end{enumerate}  
\end{defi}

\begin{defi}\label{d3.2} We will say that a process  $X=(X_t,0\leq t\leq T),$ is
  \textbf{localized by the sequence}  $\pta \Omega_k, X^k\ptc_{k\in
  \mathbb{N}^*},$ if $P\pta \cup_{k=0}^{+\infty}\Omega_k\ptc=1,$
  $\Omega_h \subseteq \Omega_k,$ if $h\leq k,$ and
  $I_{\Omega_k}X^k=I_{\Omega_k}X,$ almost surely for every $k$ in $\mathbb{N}.$
\end{defi}
\begin{rema}  \label{r4.9}
Let $(X_t, 0\leq t\leq T)$ and $(Y, 0\leq t \leq T)$ be two stochastic processes. The following statements are true.
\begin{enumerate}
\item
Let $Y$ and $X$ be localized by the
sequences $\pta \Omega_k, X^k\ptc_{k\in
  \mathbb{N}}$ and $\pta \Omega_k, Y^k\ptc_{k\in
  \mathbb{N}}$, respectively, such that $Y^k$ is $X^k$-forward integrable on $[0,T]$  for every $k$ in $\mathbb{N}.$ Then $Y$ is $X$-forward integrable on $[0,T]$ and 
$$
\int_0^\cdot Yd^-X=\int_0^\cdot
Y^kd^-X^k, \quad \mbox{ on } \Omega_k, \quad a.s..
$$
\item Given a random time $\mathcal{T} \in [0,1]$ we often denote
$X^\mathcal{T}_t = X_{t\wedge \mathcal{T}}, \ t \in [0,T]$.
\item If $Y$ is $X$-forward integrable on $[0,T],$ then $YI_{[0,\sht]}$
 is $X$-forward integrable for every random time $0\leq \sht \leq T,$ and 
$$
\int_{0}^\cdot Y_sI_{[0,t]}d^-X_s=\int_0^{\cdot \wedge t} Y_sd^-X_s. 
$$
\item If the covariation process $[X,Y]$ exists on $[0,T],$ then the
  covariation process $[X^\mathcal{T}, Y^\mathcal{T}]$
  exists for every random time $0\leq \mathcal{T} \leq T,$ and 
$$
\pqa X^\mathcal{T}, Y^\mathcal{T}\pqc = \pqa X,Y\pqc_{\mathcal{T}}. 
$$      
\end{enumerate}
 \end{rema}

\begin{defi} \label{d2.2}
Let $X=(X_t,0\leq t \leq T)$ and $Y=(Y_t,0\leq t <T)$ be processes with 
paths respectively in $C^0([0,T])$ and $L^1_{loc}([0,T)),$ i.e. $\int_0^t \vaa Y_s\vac ds< +\infty$ for any $t<T$.
\begin{enumerate}
\item If $YI_{[0,t]}$ is $X$-forward integrable for every $0\leq t < T,$ $Y$ is said \textbf{locally $X$-forward integrable on $[0,T)$}. In this case there exists a continuous process, 
which coincides, on every compact interval $[0,t]$ of $[0,1),$ with the forward integral of $YI_{[0,t]}$ with respect to $X.$ That process will still be denoted with  $I(\cdot,Y,X)=\int_0^\cdot Yd^-X.$
\item If $Y$ is locally $X$-forward integrable and  
$
\lim_{t\rightarrow T} I(t,Y,X)
$ exists almost surely, $Y$ is said $X$-\textbf{improperly forward integrable} on $[0,T]$. 
\item If the covariation process $[X^t,Y^t]$ exists, for every $0\leq t<T,$  we say that the \textbf{covariation process $[X,Y]$ exists locally} on $[0,T)$ and it is still denoted by $[X,Y].$ In this case there exists a continuous process, 
which coincides, on every compact interval $[0,t]$ of $[0,1),$ with the covariation process $\pqa X,YI_{[0,t]}\pqc.$ That process will still be denoted with  $\pqa X,Y\pqc.$ If $X=Y,$ $\pqa X,X\pqc$ we will say that the \textbf{quadratic variation of $X$ exists locally} on $[0,T].$ 
\item If the covariation process $[X,Y]$ exists locally on $[0,T)$ and $\lim_{t \rightarrow T}[X,Y]_t$ exists, the limit will be called the \textbf{improper covariation process} between $X$ and $Y$  and it will still be denoted by $[X,Y].$
If $X=Y,$ $\pqa X,X\pqc$ we will say that the \textbf{quadratic variation of $X$ exists improperly} on $[0,T].$    
\end{enumerate}
\end{defi}

\begin{rema} \label{r2.5} Let $X=(X_t, 0\leq t\leq T)$ and $ Y=(Y_t, 0\leq t\leq  T)$ be two  stochastic processes whose paths are 
 $C^0([0,T])$ and $L^1([0,T]),$ respectively. 
If $Y$ is $X$-forward integrable on $[0,T]$ then its restriction to $[0,T)$ 
is $X$-improperly forward integrable and the improper integral coincides 
with the forward integral of $Y$ with respect to $X.$ 
\end{rema}

\begin{defi} A vector $\pta \pta X^1_t,..., X^m_t\ptc,0\leq t\leq T\ptc $ of continuous processes is said to have all its \textbf{mutual brackets}
 on $[0,T]$ if $\pqa X^i,X^j\pqc$ exists on $[0,T]$ for every $i,j=1,...,m$.
\end{defi}

In the sequel if $T=1$ we will omit to specify that objects defined above exist on the interval $[0,1]$ (or $[0,1),$ respectively).

\begin{prop}\label{p3.4}
Let  $M=(M_t, 0\leq t\leq T)$ be a continuous
  local martingale with respect to some filtration $\mathbb{F}=\pta \mathcal{F}_t\ptc_{t\in[0,T]}$ of
  $\mathcal{F}.$ Then the following properties hold.
  \begin{enumerate} 
  \item The process $M$ is a finite quadratic variation process on $[0,T]$ and its quadratic variation coincides with the classical bracket appearing in the Doob decomposition of $M^2.$
\item Let $Y=(Y_t,0\leq t\leq T)$ be an $\mathbb{F}$-adapted process with left continuous and bounded paths. Then $Y$ is $M$-forward integrable on $[0,T]$
and $\int_0^\cdot Yd^-M$ coincides with the
classical It\^o integral $\int_0^\cdot YdM.$
\end{enumerate} 
\end{prop}

\begin{prop}\label{p3.6}
Let $V=(V_t,0\leq t\leq T)$  be a  bounded
  variation process and  $
Y=(Y_t,0\leq t\leq T),$ be a process with paths being bounded and with at most countable discontinuities. Then the following properties hold.
\begin{enumerate} 
\item The process $Y$ is $V$-forward
integrable on $[0,T]$ and $\int_0^\cdot Yd^-V$ coincides with the
Lebesgue-Stieltjes integral denoted with $\int_0^\cdot YdV.$
\item The covariation process $\pqa Y,V\pqc$ exists on $[0,T]$ and it is equal to zero. In particular a bounded variation process has zero quadratic variation.
\end{enumerate}
\end{prop}

\begin{coro} \label{c2.15}
Let $X=(X_t,0\leq t\leq T)$ be a continuous process and  $Y=(Y_t,0\leq t \leq T)$ a bounded variation process. Then 
$$
XY-X_0Y_0=\int_0^\cdot X_sdY_s+ \int_0^\cdot Y_sd^-X_s.
$$
\end{coro}


\begin{prop} \label{p2.1} Let $X=(X_t,0\leq t\leq T)$ be a continuous finite quadratic variation process and $V=((V^1_t,\dots,V_t^m), 0\leq t\leq T)$ be a vector of continuous bounded variation processes. Then for every $u$ in  $C^{1,2}(\mathbb{R}^m \times \mathbb{R}),$ the process $\pta  \partial_xu(V_t,X_t), 0\leq t\leq T\ptc$ is $X$-forward integrable on $[0,T]$ and
\beqn \nonu
u(V,X)&=&u(V_0,X_0)+\sum_{i=1}^m \int_0^\cdot {\partial_{v_i}} u(V_t,X_t)dV^i_t+\int_0^\cdot \partial_xu(V_t,X_t)d^-X_t\\ \nonu &+&\frac{1}{2}\int_0^\cdot \partial_{xx}^{(2)}u(V_t,X_t)d\pqa X\pqc_t.
\eeqn

\end{prop}
\begin{prop} \label{l2.1} Let  $X=(X^1_t,\dots,X^m_t, 0\leq t\leq T)$  be
 a vector of continuous processes having all its mutual brackets. 
Let  $\psi:\mathbb{R}^m\rightarrow \mathbb{R}$ be of class $C^2(\mathbb{R}^m)$
and $Y=\psi(X).$ 
Then $Z$ is $Y$-forward integrable on $[0,T],$ if and only if $Z\partial_{x^i}\psi(X)$ is $X^i$-forward integrable on $[0,T],$ for every $i=1,...,m$ and
\beqn \nonu 
\int_0^\cdot Zd^-Y=\sum_{i=1}^m \int_0^\cdot Z\partial_{x^i}\psi(X)d^-X^i+\frac{1}{2}\sum_{i,j=0}^m \int_0^\cdot Z\partial_{x^i x^j}^{(2)}\psi(X)d\pqa X^i,X^j\pqc.
\eeqn
\end{prop}
\begin{proof}
The proof derives from Proposition $4.3$ of \cite{RV2}. The result
is a slight modification of that one. It should only be noted that
there forward integral of a process $Y$ with respect to a process $X$ was defined as limit ucp 
of its regularizations.
\end{proof}
\begin{rema}\label{RL2.1}
Taking $Z = 1$, the chain rule property described in Proposition \ref{l2.1}
implies in particular the {\it classical} It\^o formula for finite
quadratic variation processes stated for instance in \cite{RV1}
or in a discretization framework in \cite{fo}.
\end{rema}

\section{  $\mathcal{A}$-martingales}
Throughout this section $\cal A$ will be a real linear space
 of measurable processes indexed by $[0,1)$ with paths which are bounded on each compact interval of $[0,1).$

We will denote with $\mathbb{F}=\pta \mathcal{F}_t\ptc_{t\in[0,1]}$ a filtration indexed by $[0,1]$ and with  $\mathcal{P}(\mathbb{F})$ the $\sigma$-algebra generated by all left continuous and $\mathbb{F}$-adapted processes. In the remainder of the paper we will  adopt the notations $\mathbb{F}$ and $\mathcal{P}(\mathbb{F})$ even when the filtration  $\mathbb{F}$ is indexed by $[0,1).$ At the same way, if $X$ is a process indexed by $[0,1],$ we shall continue to denote with  $X$ its restriction to $[0,1).$ 
\subsection{Definitions and properties}
\begin{defi} \label{d3.1}
A process $X=(X_t,\interval)$ is said
\textbf{$\mathcal{A}$-martingale} if 
every $\theta$ in $\mathcal{A}$
is $X$-improperly forward integrable 
and 
$\mathbb{E}\pqa \int_0^t \theta_sd^-X_s\pqc=0$ for every $\interval.$
\end{defi}

\begin{defi}\label{d4.2}
A process $X=(X_t,\interval)$ is said
\textbf{$\mathcal{A}$-semimartingale} if  it can be written as the sum of an $\cal{A}$-martingale $M$ and a bounded variation process $V,$ with $V_0=0.$
\end{defi}

\begin{rema}\label{r3.2}
\begin{enumerate}
\item
If $X$ is a continuous $\mathcal{A}$-martingale with $X$ belonging to $\cal{A},$ its quadratic variation exists improperly. In fact, if $\int_0^\cdot X_td^-X_t$
exists improperly, it is possible to show that $\pqa X,X\pqc$ exists improperly  and
$
\pqa X,X\pqc=X^2-X_0^2-2\int_0^\cdot X_sd^-X_s.
$
We refer to Proposition 4.1 of \cite{RV2} for details.
\item Let $X$ be a continuous square integrable martingale with respect to some filtration $\mathbb{F}$. 
Suppose that every process in $\cal{A}$ is the restriction to $[0,1)$ of a process $(\theta_t,\interval)$ which is 
$\mathbb{F}$-adapted, it has  left continuous with right limit paths 
(cadlag) and 
$\mathbb{E}\pqa \int_0^1 \theta_t^2 d\pqa X \pqc_t\pqc<+\infty.$ 
Then $X$ is an $\mathcal{A}$-martingale.  

\item In \cite{FWY} the authors introduced the notion of
  \textbf{weak-martingale}. 
   A semimartingale  $X$ is a
  weak-martingale if $\mathbb{E}\pqa \int_0^t f(s,X_s)dX_s
  \pqc=0,$ $\interval,$ for every $f:\mathbb{R}^+\times \mathbb{R}\rightarrow \mathbb{R},$ bounded
  Borel-measurable. Clearly we can affirm the following.  Suppose that
  $\mathcal{A}$ contains all processes of the form $f(\cdot,X),$ with
  $f$ as above. Let $X$ be a semimartingale 
 which is an $\mathcal{A}$-martingale. Then $X$ is a weak-martingale. 
\end{enumerate}
\end{rema}

\begin{prop} \label{p3.4a}
Let  $X$ be a continuous $\mathcal{A}$-martingale. The following statements hold true.
\begin{enumerate}
\item If $X$ belongs to $\cal{A},$ $X_0=0$ and $\pqa X,X\pqc=0.$ Then $X\equiv 0$.
\item Suppose that $\mathcal{A}$ contains all bounded  $\mathcal{P}(\mathbb{F})$-measurable processes.
Then $X$ is an $\mathbb{F}$-martingale.
\end{enumerate}
\end{prop}
\begin{proof}
From point 1. of Remark \ref{r3.2}, $\mathbb{E}\pqa X_t^2\pqc=0,$ for all $\interval.$  

Regarding point 2. it is sufficient to observe that processes of type $I_{A}I_{(s,t]},$ with $0  \leq s \leq t \leq 1,$ and $A$ in $\mathcal{F}_s$ belong to $\cal A.$ Moreover $\int_0^1 I_AI_{(s,t]}(r)d^-X_r=I_A(X_t-X_s).$ This imply $\mathbb E [X_t-X_s\left|\right.\mathcal{F} _s]=0,$ $0\leq s \leq t \leq 1.$ 
\end{proof}

\begin{coro}
The decomposition of an $\cal A$-semimartingale $X$ in definition \ref{d4.2} is unique among the class of processes of type $M+V,$ being $M$ a continuous $\mathcal{A}$-martingale in $\cal A$ and $V$ a bounded variation process.
\end{coro}
\begin{proof}
If $M+V$ and $N+W$ are two decompositions of that type, then $M-N$ is a continuous $\mathcal{A}$-martingale in $\cal A$ starting at zero with zero quadratic variation.  Point 1. of Proposition \ref{p3.4a} permits to conclude.
\end{proof}

The following Proposition gives sufficient conditions for an $\cal{A}$-martingale to be a martingale with respect to some filtration $\mathbb{F},$ when $\cal{A}$ is made up of $\mathcal{P}(\mathbb{F})$-measurable processes. It constitutes a generalization of point 2. in Proposition \ref{p3.4a}.

\begin{defi}\label{d3.6}
We will say that $\mathcal{A}$ satisfies \textbf{Assumption $\mathcal{D}$} with respect to a filtration $\mathbb{F}$ if 
\begin{enumerate}
\item Every $\theta$ in $\mathcal{A}$ is $\mathbb{F}$-adapted;
\item For every $0\leq s <1$ there exists a basis  $\mathcal{B}_s$ for $\mathcal{F}_s,$ with the following property. For every $A$ in $\mathcal{B}_s$ there exists a sequence of $\mathcal{F}_s$-measurable random variables $(\Theta_n)_{n\in \mathbb{N}},$ such that for each $n$ the process $\Theta_nI_{[0,1)}$ belongs to $\mathcal{A},$ $\sup_{n\in \mathbb{N}}\vaa \Theta_n\vac\leq 1,$ almost surely and
$$
\lim_{n\rightarrow +\infty}\Theta_n=I_{A}, \quad a.s.
$$
\end{enumerate}
\end{defi}

\begin{prop}\label{p3.3}
Let $X=(X_t,\interval)$ be a continuous {$\mathcal{A}$-martingale} adapted to some filtration $\mathbb{F}$, with $X_t$ belonging to $L^1(\Omega)$ for every $\interval.$ Suppose that $\mathcal{A}$ satisfies Assumption $\mathcal{D}$ with respect to $\mathbb{F}.$ Then $X$ is an \filtrationf-martingale.
\end{prop}

\begin{proof}
We have to show that for all $0\leq s\leq t\leq 1$,
$\mathbb{E}\pqa I_{A}\pta X_t-X_s\ptc
\pqc=0$, for all  $A$ in $\mathcal{F}_s$.
We fix $0\leq s<t\leq  1$ and $A$ in $\mathcal{B}_s.$  Let $(\Theta_n)_{n\in \mathbb{N}}$ be a sequence of random variables converging almost surely to $I_{A}$ as in the hypothesis. Since $X$ is an $\cal{A}$-martingale,
$\mathbb{E}^{}\pqa \Theta_n \pta X_t-X_s\ptc
\pqc=0$, for all $n$ in $\mathbb{N}$.
We note that $X_t-X_s$ belongs to $L^1(\Omega),$ then, by Lebesgue dominated convergence theorem,
\beqn\nonu
\vaa \mathbb{E}\pqa I_{A}\pta
X_t-X_s\ptc\pqc\vac\leq\lim_{n\rightarrow +\infty}\mathbb{E}\pqa
\vaa I_{A}-\Theta_n \vac \vaa X_t-X_s \vac\pqc=0.
\eeqn
Previous result extends to the whole $\sigma$-algebra $\mathcal{F}_s$ and this permits to achieve the end of the proof.  
\end{proof}

Some interesting properties can be derived taking inspiration from
\cite{FWY}.

For a process $X,$ we will denote
\beqn \label{AX}
\mathcal{A}_X&=&\pga (\psi(t,X_t)), \intervala
\left |\right. \psi:[0,1]\times \mathbb{R} \rightarrow \mathbb{R}, \mbox{ Borel-measurable }\right. \\ \nonu &&\left.
\mbox{ with polynomial growth
}\pgc.
\eeqn
  

\begin{prop}
Let $X$ be a continuous $\mathcal{A}$-martingale with $\mathcal{A}=\mathcal{A}_X.$ 

Then, for every $\psi$ in $C^2(\mathbb R)$ with bounded first and second derivatives, the process 
$$
\psi(X)-\frac{1}{2}\int^{\cdot}_{0}\psi''(X_s)d\pqa X,X\pqc_s
$$
is an $\mathcal{A}$-martingale.
\end{prop}
\begin{proof}
The process $X$ belongs to $\cal A.$ In particular, $X$ admits improper quadratic variation. We set $Y=\psi(X)-\frac{1}{2}\int^{\cdot}_{0}\psi''(X_s)d\pqa X,X\pqc_s.$
Let $\theta$ in $\mathcal {A}_X.$ By Proposition \ref{l2.1}, for every $0\leq t < 1$
$$
\int^{t}_{0}\theta_s d^-Y_s =\int^{t}_{0}\theta_s\psi'(X_s)d^-X_s.
$$
Since $\theta \psi'(X)$ still belongs to $\cal A,$ $\theta$ is
$Y$-improperly
forward  integrable and  
\beqn \label{e8}
\int^{\cdot}_{0}\theta_t d^-Y_t =\int^{\cdot}_{0}\theta_t\psi'(X_t)d^-X_t.
\eeqn
We  conclude taking the expectation in equality (\ref{e8}). 
\end{proof}

\begin{prop}\label{p4.9}
Suppose that $\cal A$ is an algebra. Let $X$ and $Y$ be two continuous $\mathcal{A}$-martingales with $X$ and $Y$ in $\cal{A}.$ 

Then the process $XY-\pqa X,Y\pqc$ is an $\mathcal{A}$-martingale . 
\end{prop}

\begin{proof}
Since $\cal A$ is a real linear space, $(X+Y)$ belongs to $\cal A.$ In particular by point 1. of Remark \ref{r3.2}, $\pqa X+Y,X+Y\pqc,$ $\pqa X,X\pqc$ and $\pqa Y,Y\pqc$ exist improperly. 
This implies that $\pqa X,Y\pqc$ exists improperly too and that it is a bounded variation process. 
Therefore the vector $(X,Y)$ admits all its mutual brackets on each compact set of $[0,1).$ 
Let $\theta$ be in $\mathcal{A}.$ 
Since
$\mathcal{A}$ is an algebra, $\theta X$ and $\theta Y$ belong to $\mathcal{A}$ and so both $\int_0^\cdot  \theta_s X_s d^-Y_s$ and $\int_0^\cdot \theta_s Y_s d^-X_s$ locally exist.  
By Proposition \ref{l2.1} $\int_0^\cdot \theta_t d^-\pta X_tY_t-\pqa X,Y\pqc\ptc $ exists improperly too  and
\beqn \nonu
\int_0^\cdot \theta_td^-\pta X_tY_t-\pqa X,Y\pqc_t\ptc=\int_0^\cdot Y_t\theta_t d^-X_t+\int_0^\cdot  X_t\theta_t  d^-Y_t. 
\eeqn
Taking the expectation in the last expression we then get the result.
\end{proof}

We recall a notion and a related result of \cite{CR}. 

A process $R$ is {\textit{{strongly predictable}}} with respect to a filtration $\mathbb{F},$ if 
$$
\exists\ \delta>0, \mbox{ such that } R_{\ep+\cdot} \mbox{ is } \mathbb{F}\mbox{-adapted},\ \mbox{ for every } \ep\leq\delta.
$$

\begin{prop} \label{p3.11}
Let $R$ be an $\mathbb{F}$-strongly predictable continuous process. Then for every continuous $\mathbb{F}$-local martingale $Y,$ $\pqa R,Y\pqc=0$. 
\end{prop}

Proposition \ref{p3.11} combined with  Proposition \ref{p4.9} implies Proposition \ref{p3.10} and Corollary \ref{c3.11}.

\begin{prop}\label{p3.10}
Let $\cal{A},$ $X$ and $Y$ be as in Proposition \ref{p4.9}. Assume, moreover, that $X$ is an $\mathbb{F}$-local martingale, and that  $Y$ is strongly predictable with respect to $\mathbb{F}.$
Then $XY$ is an $\mathcal{A}$-martingale.
\end{prop}

\begin{coro}\label{c3.11}
Let $\cal{A},$ $X$ and $Y$ be as in Proposition \ref{p4.9}.
Assume that $X$ is a local martingale with respect to some
filtration $\mathbb{G}$ and that $Y$ is either $\mathbb{G}$-independent, or $\mathcal{G}_0$-measurable. 
Then $XY$ is an $\mathcal{A}$-martingale.
\end{coro}
\begin{proof}
If $Y$ is $\mathbb{G}$-independent, it is sufficient to apply
 previous Proposition with \\
$\mathbb{F}=\pta \bigcap_{\ep>0} \mathcal{G}_{t+\ep}\vee 
\sigma(Y)\ptc_{t\in[0,1]}$. Otherwise one takes $\mathbb{F} = \mathbb{G}$.

\end{proof}

\subsection{$\mathcal{A}$-martingales and Weak Brownian motion}
We proceed defining and discussing processes which are \textit{weak-Brownian motions} in order to exhibit explicit examples of $\mathcal{A}$-martingales.

\begin{defi}(\cite{FWY})		\label{defKOBM}
A stochastic process $(X_t, \interval)$ is a \textbf{weak Brownian
  motion of order $k$} if for every k-tuple $(t_1,t_2,...,t_k)$
$$
(X_{t_1},X_{t_2},\dots,X_{t_k})\stackrel{law}{=}(W_{t_1},W_{t_2},\dots,W_{t_k})
$$
where $(W_t,\interval)$ is a Brownian motion.
\end{defi}
\begin{rema}\label{Rpoint2}
\begin{enumerate} 
\item Using the definition of quadratic variation it is not
difficult to show for a weak Brownian motion of 
order $k \ge 4$,
we have $[X]_t = t$.
\item In \cite{FWY} it is shown that for any $k \ge 1$, there exists 
a weak $k$-order Brownian motion which is different from classical
Wiener process.
\item If $k \ge 2$ then $X$ admits a continuous modification
and can be therefore always considered continuous.
\end{enumerate} 
\end{rema}
For a process $(X_t, \interval),$ we set
 \beqn \nonu
\mathcal{A}^1_{X}&=&\pga (\psi(t,X_t), 
\interval, \mbox{ with polynomial growth  s.t }
 \psi=\partial_x \Psi \right. \\ \nonu && \left.\Psi \in
 C^{1,2}([0,1]\times \mathbb{R}) \mbox{ with  }
  \partial_t\Psi  \mbox{ and }   \partial^{(2)}_{xx}\Psi
 \mbox{ bounded.}
\pgc.
\eeqn
\begin{assu}\label{a4.15}
Let $\sigma:[0,1]\times
\mathbb{R}\rightarrow \mathbb{R}$ be a 
Borel-measurable and bounded function.
We suppose moreover that the
 equation 
\beqn \label{e11}
\left\{
\begin{array}{ll}
\partial_t \nu_t(dx)=\frac{1}{2}\partial^{(2)}_{xx}\pta \sigma^2(t,x)\nu_t(dx)\ptc \\
\nu_0(dx)=\delta_{0}.
\end{array} \right.
\eeqn
admits a unique solution  
 $\pta \nu_t\ptc_{t\in [0,1]}$ in the sense of distributions, 
in the class of continuous functions $t \mapsto \shm(\R)$
where $\shm(\R)$ is the linear space of finite signed Borel 
real measures, equipped with the weak topology.
\end{assu}

\begin{rema} \label{Rsv}
\begin{enumerate}
\item Assumption \ref{a4.15} is verified for 
$\sigma(t,x)\equiv \sigma,$
 being  $\sigma$ a positive real constant and,  in that case,
 $\nu_t=N(0,\sigma^2t),$ for every $\interval.$   
\item Suppose moreover the following.
 For every compact set of 
$[0,1] \times \R$ $\sigma$ is lower bounded by a positive constant.
We say in this case that $\sigma$ is {\bf non-degenerate}. \\ 
Exercises 7.3.2-7.3.4 of \cite{sv} (see also \cite{ks1},
 Refinements 4.32, Chap. 5) say that there is a weak unique 
solution to equation
$
dZ=\sigma(\cdot,Z)dW, Z_0=0,
$
$W$ being a classical Wiener process.
By a simple application of It\^o's formula,
the law $(\nu_t(dx))$ of $Z_t$ provides a solution
to \eqref{e11}. \\
According to Exercise 7.3.3 of \cite{sv} 
(Krylov estimates) it is possible to show the existence
of $(t,x) \mapsto p(t,x)$ in $L^2[[0,1] \times \R)$ being 
density of $(\nu_t(dx))$. In particular
for almost all $t \in [0,1]$, 
$\nu_t(dx)$ admits a density.
\end{enumerate}
\end{rema}  

\begin{prop}\label{p}
Let $(X_t,\interval)$ be a continuous finite quadratic variation
 process with $X_0=0,$  and
$d\pqa X\pqc_t=(\sigma(t,X_t))^2 dt,$ where $\sigma$ fulfills 
Assumption \ref{a4.15}.  Suppose that $\mathcal A=\mathcal{A}^1_X.$
Then the following statements are true.
\begin{enumerate}
\item \label{point1}
$X$ is an $\mathcal{A}$-martingale if and only if, for every 
  $\interval,$ $X_t\stackrel{law}{=}Z_t,$ for every 
$(Z,B)$ solution of
  equation
$
dZ=\sigma(\cdot,Z)dB, Z_0=0.
$
In particular, if
$\sigma\equiv 1,$ $X$ is a weak Brownian motion of order $1$, if and only if it is an $\mathcal{A}_X^{1}$-martingale.
\item \label{point2}
Suppose that $d\pqa X\pqc_t=f_tdt,$ with $f$  being $\mathcal{B}([0,1])$-measurable and bounded. If  $X$ is a weak Brownian motion
of order $k=1,$ then $X$ is an $\mathcal{A}$-semimartingale. Moreover the process 
$$
X+\int_0^\cdot \frac{(1-f_s)X_s}{2s} ds.
$$
is an $\mathcal{A}$-martingale.
\end{enumerate}
\end{prop}
\begin{proof}
\begin{enumerate}
\item
Using It\^o's  formula recalled in Proposition \ref{p2.1} we can write, for
every $\interval$ and  $\psi=\partial_x\Psi$ 
according to the definition of  $\mathcal{A}^1_X$
\begin{eqnarray}  \label{f2}
\int_0^t \psi(s,X_s)d^-X_s&=&\Psi(t,X_t)-\Psi(0,X_0) \nonu \\
&& \\
 &-&  \int_0^t \pta \partial_s\Psi
+\frac{1}{2}\partial_{xx}^{(2)}\Psi\sigma^2\ptc(s,X_s) ds. \nonu
\end{eqnarray}
For every $\interval,$ we denote with $\mu_t(dx)$ the law of $X_t.$ If $X$ is an $\mathcal{A}^1_X$-martingale, from (\ref{f2}) we derive      
\beqn \nonu
0&=&\int_\mathbb{R} \Psi(t,x)\mu_t(dx)-\int_\mathbb{R}  \Psi(0,x)\mu_0(dx) -\int_0^t\int_\mathbb{R}
\partial_s \Psi(s,x)\mu_s(dx) ds\\ \label{e3} &-&\frac{1}{2}\int_0^t \int_\mathbb{R}
\partial_{xx}^{(2)}\Psi(s,x)\sigma(s,x)^2 \mu_s(dx) ds.
\eeqn

In particular,  the law of $X$ solves equation (\ref{e11}).

On the other hand, let
$(Z,B)$ be a solution of equation
$
Z=\int_0^\cdot \sigma(s,Z_s)dB_s.
$ 
The law of $Z$ fulfills equation (\ref{e3}) too. Indeed, $Z$ is a finite quadratic variation process with $d\pqa Z\pqc_t=(\sigma(t,Z_t))^2dt$ which is an $\mathcal{A}^1_X$-martingale by point 2. of Remark \ref{r3.2}.
 By Assumption \ref{a4.15} $X_t$ must have the same law as $Z_t.$ This establishes the direct implication of point 1.

Suppose, on the contrary, that $X_t$ has the same law as
 $Z_t,$ for every $\interval.$ Using the fact that $Z$ is an $\mathcal{A}^1_X$-martingale which solves equation (\ref{f2}) we get
$$
\mathbb{E}\pqa \Psi(t,Z_t)-\Psi(0,Z_0)-\int_0^t \pta \partial_s\Psi 
+\frac{1}{2}\partial_{xx}^{(2)}\Psi \sigma^2\ptc (s,Z_s) ds \pqc=0,
$$
for every $\Psi$ in $C^{1,2}([0,1]\times \mathbb{R})$ with $\partial_x
\Psi=\psi$ according to $\mathcal{A}^1_X.$
Since $X_t$ has the same law as $Z_t,$ for every $\interval,$ equality (\ref{f2}) implies that  $$
\mathbb{E}\pqa \int_0^\cdot \psi(t,X_t)d^-X_t\pqc=\mathbb{E}\pqa \int_0^\cdot \psi(t,Z_t)d^-Z_t\pqc=0, 
$$
The proof of the first point is now  achieved.
\item
Suppose that $\sigma(t,x)^2=f_t,$ for every $(t,x)$ in $[0,1]\times
\mathbb{R}.$ Let $\Psi$ be in $C^{1,2}\pta [0,1]\times \mathbb{R}\ptc$
such that $\psi(\cdot,X) =\partial_x \Psi (\cdot,X)$ belongs to $\mathcal{A}^1_X.$ Proposition \ref{p2.1} yields 
$$
\int_0^t \psi(s,X_s)d^-X_s=Y^\Psi_t+\frac{1}{2}\int_0^t
\partial_{xx}^{(2)}\Psi(s,X_s)(1-f_s)ds, \quad \interval,
$$
with
$$
Y^\Psi_t=\Psi(t,X_t)-\Psi(0,X_0)-\int_0^t
\partial_{s}\Psi_s(s,X_s)ds-\frac{1}{2}\int_0^t \partial_{xx}^{(2)}\Psi(s,X_s)ds.
$$
Moreover $X$ is a weak Brownian motion of order
$1.$ This implies 
$
\mathbb{E}\pqa Y^\Psi_t \pqc=0$, for every $\interval$.
We derive that 
$$
\mathbb{E}\pqa \int_0^t \psi(s,X_s)d^-X_s+\frac{1}{2}\int_0^t
\partial_{xx}^{(2)}\Psi(s,X_s)(f_s-1)ds\pqc=\mathbb{E}\pqa Y^\Psi_t \pqc=0.
$$
Since the law of $X_t$ is $N(0,t),$ by Fubini's theorem and integration by parts on the real line we obtain 
$$
\mathbb{E}\pqa \int_0^t
\partial_{xx}^{(2)}\Psi(s,X_s)(f_s-1)ds\pqc=\mathbb{E}\pqa \int_0^t
\psi(s,X_s)\frac{(1-f_s)X_s}{s}ds\pqc.
$$
This concludes the proof of the second point.
\end{enumerate}
\end{proof}
\begin{rema} \label{RBRR}
In the statement of Proposition
\ref{p}, we may not suppose a priori the uniqueness for PDE \eqref{e11}.
We can replace it with the following. 
\begin{assu} \label{a4.15bis}
\begin{itemize} 
\item $\sigma$ is non-degenerate.
\item 
Let $\mu_t(dx)$ be the law of $X_t, t \in [0,1]$. We suppose that the Borel 
finite  measure $\mu_t(dx) dt$ on $[0,1] \times \R$
admits a density $(t,x) \mapsto  q(t,x)$ in $L^2([0,1] \times \R)$. 
\end{itemize} 
\end{assu}
In fact, the same proof as for item 1. works, taking into account
item 2. of Remark \ref{Rsv} the difference $p-q$ belongs to
$L^2([0,1] \times \R)$; by Theorem 3.8 of \cite{BRR} $p = q$
and so the law of $X_t$ and $Z_t$ are the same for any $t \in[0,1]$.
\end{rema}
From \cite{FWY} we can extract an example of an $\cal{A}$-semimartingale 
 which is not a semimartingale.

\begin{exam}
Suppose that $\pta B_t, \interval\ptc$ is a Brownian motion on the
 probability space $\pta \Omega, \mathbb{G},P\ptc,$ being $\mathbb{G}$ 
some filtration on $(\Omega, \mathcal{F},P).$ Set
$$
X_t=\pga
\begin{array}{ll}
B_t,& 0\leq t\leq \frac{1}{2} \\
B_{\frac{1}{2}}+(\sqrt{2}-1)B_{t-\frac{1}{2}}, & \frac{1}{2}< t\leq 1.
\end{array}
\right.
$$
Then $X$ is a continuous weak Brownian motion of order $1,$ which 
is not a $\mathbb{G}$-semimartingale. Moreover it is possible to show that
 $d\pqa X \pqc_t=f_tdt,$ with 
$f=I_{[0,\frac{1}{2}]}+(\sqrt{2}-1)^2I_{[\frac{1}{2},1]}.$
In particular, thanks to  point 2. of previous Proposition \ref{p},
$X+\int_0^\cdot \frac{(1-f_s)X_s}{2s}ds$  is an
${\cal{A}}^1_X$-martingale. In fact the notion of quadratic variation
is not affected by the enlargement of filtration.
\end{exam}

A natural question is the following. Supposing that $X$ is an $\cal{A}$-martingale with respect to a probability measure $Q$ equivalent to $P,$ what can we say about the nature of $X$ under $P$?  The following Proposition provides a partial answer to this problem when $\mathcal{A}=\mathcal{A}_X^1.$
  
\begin{prop} 		\label{p318}
Let $X$ be as in Proposition \ref{p}, and $\sigma$  satisfy Assumption 
\ref{a4.15}. Assume, furthermore,  that  $X$ is an  $\mathcal{A}^1_X$-martingale under a
probability measure $Q$ with $P << Q.$ 
Suppose that the solution $(\nu_t(dx))$ of \eqref{e11} admits a density for every $t\in (0,1]$. 
Then the law of $X_t$ is absolutely continuous with respect to
Lebesgue measure, for  all $t\in (0,1]$.
\end{prop}
\begin{proof}
Since $P << Q,$ for every $\interval,$
the law of $X_t$ under $P$ is absolutely continuous with
respect to the law of $X_t$ under $Q$. 
 Then  it is sufficient to observe that by Proposition \ref{p}, for all
  $\interval,$ the law of $X_t$ under $Q$ is absolutely continuous
 with respect to Lebesgue. 
By Proposition \ref{p}, the law of $X_t$ is equivalent to the law $\nu_t$ of $Z_t$ 
for every $t\in [0,1]$. The conclusion follows because $\nu_t$ is absolutely continuous.
\end{proof}
  
\begin{coro} \label{c4.16}
Let $X$ be as in Proposition \ref{p}, and $\sigma$  satisfy Assumption 
\ref{a4.15}. Assume, furthermore,  that  $X$ is an  $\mathcal{A}_X$-martingale under a
probability measure $Q$ with $P << Q,$   
Then the law of $X_t$ is absolutely continuous with respect to
Lebesgue measure, for every $\interval.$
\end{coro}  
\begin{proof}
Clearly $\mathcal{A}^1_X$ is contained in $\mathcal{A}_X.$ The result is then a consequence of previous Proposition \ref{p318}. 
\end{proof}
\begin{prop}
Let $\pta X_t,\interval\ptc$ be a continuous weak Brownian motion of
order $ 8.$ Then, for every $\psi:[0,1]\times \mathbb{R} \rightarrow \mathbb{R}$, Borel measurable with polynomial growth,  
the forward integral $\int_0^\cdot \psi(t,X_t)d^-X_t,$ exists and
$$
\mathbb{E}\pqa \int_0^\cdot\psi(t,X_t)d^-X_t, \pqc=0.
$$
In particular, $X$ is an $\mathcal{A}_X$-martingale.
\end{prop}

\begin{proof}
Let $\psi:[0,1]\times \mathbb{R} \rightarrow \mathbb{R}$ be Borel measurable and  $t$ in $\interval$ be fixed. Set
$$
I_\ep^X(t)=I(\ep,\psi(\cdot,X),X)
\quad I_\ep^B(t)=I(\ep,\psi(\cdot,B),B),
$$
being $B$ a  Brownian motion on a filtered probability space $(\Omega^B, \mathbb{F}^B,P^B).$ 

Since $X$ is a weak Brownian motion of order $8,$ it follows that  
$$
\mathbb{E}\pqa \vaa I_\ep^X(t)-I_\delta^X(t)\vac^4\pqc=\mathbb{E}^{P^B}\pqa \vaa I_\ep^B(t)-I_\delta^B(t)\vac^4\pqc, \quad  \forall \ \ep, \delta>0.
$$
We show now that $I^B_\ep(t)$ converges in $L^4(\Omega).$ This implies that $I^X_\ep(t)$ is of Cauchy in $L^4(\Omega).$

In \cite{RV05}, chapter 3.5, it is proved that $I_\ep^B(t)$ converges in probability when $\ep$ goes to zero, and the limit equals the It\^o integral $\int_0^t \psi(s,B_s)dB_s.$
Applying Fubini's theorem for It\^o integrals, theorem 45 of \cite{Protter},
 chapter IV and  Burkholder-Davies-Gundy inequality, 
we can perform the following estimate, for every $p>4:$
\beqn \nonu
\mathbb{E}^{P^B}\pqa \vaa I_\ep^B(t)\vac^p\pqc 
\leq  c \sup_{t\in[0,1]}\mathbb{E}^{P^B}\pqa \vaa \psi(t,B_t)
\vac^p\pqc< +\infty,
\eeqn
for some positive constant $c.$
This implies the uniformly integrability of the family of random variables $\pta  (I_\ep^B(t))^4 \ptc_{\ep>0}$ and therefore the convergence in $L^4(\Omega^B,P^B)$ of $\pta I_\ep^B(t)\ptc_{\ep>0}.$

Consequently,  $\pta  I_\ep^X(t) \ptc_{\ep>0}$ converges in $L^4(\Omega)$ toward a random variable $I(t).$ It is clear that $\mathbb{E}\pqa I(t)\pqc=0,$ being $I(t)$ the limit in $L^2(\Omega)$ of random variables having zero expectation. 

To conclude we show that Kolmogorov lemma applies to find a continuous version of $\pta I(t), \interval\ptc.$ Let $0\leq s\leq t\leq 1.$
Applying the same arguments used above
\beqn \nonu
\mathbb{E}\pqa \vaa I(t)-I(s)\vac^4\pqc
\leq  \sup_{u\in[0,1]} \mathbb{E}^{P^B}\pqa
  \vaa \psi(u,B_u)\vac ^4 \pqc \vaa t-s \vac^2, \quad c>0.
\eeqn

\end{proof}

\begin{rema}\label{RFWYQV}
If $X$ is a $4$-order weak Brownian motion than,
using the techniques of proof of previous result,
that $W$ has quadratic variation 
$[X]_t = t$.

\end{rema}

\subsection{Optimization problems and  $\cal{A}$-martingale property}
\subsubsection{G\^ateaux-derivative: recalls}
In this part of the paper we recall the notion of G\^ateaux differentiability and we list
some related properties.
\begin{defi}
A function $f:\cal{A}\rightarrow \mathbb{R}$ is said
\textbf{{G\^ateaux-differentiable}} at $\pi \in \cal{A},$ if there
exists $D_{\pi} f:\cal{A}\rightarrow \mathbb{R}$ such that
$$
\lim_{\ep \rightarrow
  0}\frac{f(\pi+\ep\theta)-f(\pi)}{\ep}=D_{\pi} f(\theta),\quad \forall
\theta \in \cal{A}.
$$
If $f$ is \textit{G\^ateaux}-differentiable at every $\pi \in
\cal{A},$ then $f$ is said \textit{G\^ateaux}-differentiable on
$\cal{A}$.
\end{defi}

\begin{defi} \label{d3.18}
Let  $f:\cal{A}\rightarrow \mathbb{R}.$ A process $\pi$ is said
\textbf{optimal} for $f$ in $\mathcal{A}$ if
$$
f(\pi)\geq f(\theta), \quad \forall \theta \in \mathcal{A}.
$$
\end{defi}

We state this useful lemma omitting its straightforward proof.

\begin{lemm}\label{l7.2}
Let  $f:\cal{A}\rightarrow \mathbb{R}.$ For every $\pi$ and $\theta$
in $\cal{A}$ define $f_{\pi,\theta}:
\mathbb{R}\longrightarrow\mathbb{R}$ in the following way:
$$
f_{\pi,\theta}(\lambda)=f(\pi+\lambda(\theta-\pi)).
$$
Then it
holds:
\begin{enumerate}
\item f is \textit{G\^ateaux}-differentiable if and only if for every $\pi$ and $\theta$ in $\cal{A},$  $f_{\pi,\theta}$ is differentiable
  on $\mathbb{R}.$ Moreover
  $f_{\pi,\theta}'(\lambda)=D_{\pi+\lambda(\theta-\pi)} f (\theta-\pi).$
\item  f is concave if and only if  $f_{\pi,\theta}$ is concave
  for every  $\pi$ and $\theta$ in $\cal{A}$.
\end{enumerate}
\end{lemm}

\begin{prop}\label{p7.3}
Let  $f:\cal{A}\rightarrow \mathbb{R}$ be
\textit{G\^ateaux}-differentiable.  Then, if $\pi$ is optimal for $f$
in $\mathcal{A}, then $ $D_\pi f=0.$ If $f$ is concave
$$
\pi \mbox{ is optimal for } f \mbox{ in }\mathcal{A}  \quad
\Longleftrightarrow  \quad D_\pi f =0.
$$
\end{prop}

\begin{proof}
It is immediate to prove that $\pi$ is optimal for $f$ if and only if $\lambda=0$ is a maximum for $f_{\pi,\theta},$ for every $\theta$ in $\cal{A}.$ By Lemma \ref{l7.2} $f^{'}_{\pi,\theta}(0)=D_{\pi} f(\theta),$ for every $\theta$ in $\cal{A}$. The conclusion follows easily.
\end{proof}
\subsubsection{An optimization problem}\label{s3.3.2}
 In this part of the paper $F$ will be supposed to be a measurable function on $(\Omega\times \mathbb{R},\mathcal{F}\otimes \mathcal{B}(\mathbb{R})),$ almost surely in $C^1(\mathbb{R}),$ strictly increasing, with $F'$ being the derivative of $F$ with respect to $x,$ bounded on $\mathbb{R}$, uniformly in $\Omega.$
In the sequel $\xi$ will be a continuous  finite quadratic variation process with $\xi_0=0.$

The starting point of our construction is the following hypothesis.

\begin{assu}\label{a0}
\begin{enumerate}
\item If $\theta$ belongs to $\cal{A},$ then $ \theta I_{[0,t]}$ belongs to $\mathcal{A}$ for every $\intervala.$
\item  Every $\theta$ in $\cal{A}$ is $\xi$-improperly forward integrable, and
$$
\mathbb{E}\pqa \vaa \int_0^1 \theta_t d^-\xi_t\vac +\vaa \int_0^1 \theta_t^2
 d[\xi]_t \vac \pqc < +\infty.
$$
\end{enumerate}
\end{assu}
\begin{defi} 
Let $\theta$ be  in $\cal{A}.$ We denote 
$$
L^\theta=\int_0^1 \theta_td^-\xi_t-\frac{1}{2}\int_0^1 \theta^2_t d[\xi]_t,
\quad 
dQ^\theta=\frac{F'(L^\theta)}{\mathbb{E}\pqa F'(L^\theta)\pqc}
$$ and 
we set $f(\theta)=\mathbb{E}\pqa F(L^\theta) \pqc.$ 
\end{defi}

We observe that point 2. of Assumption \ref{a0} and the boundedness of $F'$ imply that  $\mathbb{E}\pqa \vaa F(L^\theta) \vac\pqc <+\infty.$ Therefore $f$ is well defined.
 \begin{rema} Point 2. of Assumption \ref{a0} implies that $\mathbb{E}\pqa \vaa \xi_t \vac+\pqa \xi \pqc_t\pqc < +\infty,$ for every $\interval.$ This is due to the fact that $\cal A$ must contain real constants. 
 \end{rema}
We are interested in describing a link between the existence of an optimal process for $f$ in $\cal{A}$ and the $\mathcal{A}$-semimartingale property for $\xi$ under some probability measure equivalent to $P,$ depending on the optimal process.

\begin{lemm}\label{l5.1}
The function $f$ is \textit{G\^ateaux}-differentiable on
$\cal{A}.$ Moreover for every $\pi$ and $\theta$ in
$\mathcal{A}$
$$
D_\pi f(\theta)=\mathbb{E}\pqa F'(L^\pi)\int_0^1
\theta_t d^-\pta \xi_t- \int_0^t \pi_s d\pqa \xi\pqc_s \ptc \pqc.
$$
If $F$ is concave, then $f$ inherits the property.
\end{lemm}

\begin{proof}
Regarding the concavity of $f,$ we recall that if  $F$ is increasing and concave, it is sufficient to verify that, for every $\theta$ and $\pi$ in $\mathcal{A}$, it holds
$$
L^{\pi+\lambda(\theta-\pi)}-L^{\pi}-\lambda\pta L^{\theta}-L^{\pi}\ptc\geq 0, \quad 0\leq \lambda \leq 1.
$$
A short calculation shows that, 
for every $0\leq \lambda \leq 1,$
$$
L^{\pi+\lambda(\theta-\pi)}-L^{\pi}-\lambda\pta L^{\theta}-L^{\pi}\ptc=\frac{1}{2} \lambda(1-\lambda)\int_0^1 (\theta_t-\pi_t)^2d\pqa\xi \pqc_t\geq 0.
$$
Using the
differentiability of $F$ we can write
$$
a_\ep=\frac{1}{\ep}(f(\pi+\ep\theta)-f(\pi))=\mathbb{E}\pqa
H^\ep_{\pi,\theta}\int_0^1 F'\pta
L^{\pi}+\mu \ep H^\ep_{\pi,\theta} \ptc
d\mu  \pqc,
$$
with
\beqn \nonu
H^{\ep}_{\pi,\theta}
=\int_0^1 \theta_td^-\xi_t-\frac{1}{2}\int_0^1
(\theta_t^2\ep+2\theta_t\pi_t)d\pqa \xi\pqc_t.
\eeqn
The conclusion follows by Lebesgue dominated convergence theorem, which applies thanks to the boundedness of $F'$ and point 2. in Assumption \ref{a0}.
\end{proof}

Putting together Lemma \ref{l5.1} and Proposition \ref{p7.3} we can
formulate the following.

\begin{prop}\label{p7.8}
If a process $\pi$ in $\cal{A}$ is optimal for  
$
\theta\mapsto \mathbb{E}\pqa F\pta 
L^\theta \ptc \pqc,
$ then  the process
$
\xi-\int_0^\cdot \pi_t d \pqa \xi\pqc_t
$ is an  $\cal{A}$-martingale under $Q^\pi.$ 
If $F$ is concave the converse holds.
\end{prop}

\begin{proof}
Thanks to Lemma \ref{l5.1} and point 1. in  Assumption \ref{a0}, for every $\theta$ in $\mathcal{A}$ and $\interval$
\beqn \nonu
0&=&D_\pi f(\theta I_{[0,t]})=\mathbb{E}\pqa F'(L^\pi)\int_0^t
\theta_s d^-\pta \xi_s-\int_0^s \pi_r d\pqa \xi\pqc_r \ptc \pqc \\ \nonu
&=&\mathbb{E}^{Q^\pi}\pqa \int_0^t
\theta_s d^-\pta \xi_s-\int_0^s \pi_r d\pqa \xi\pqc_r \ptc \pqc.
\eeqn
\end{proof}

The following Proposition describes some sufficient conditions to recover the semimartingale property for $\xi$ with respect to a filtration $\mathbb{G}$ on $(\Omega,\mathcal{F}), $ when the set $\mathcal{A}$ is made up of $\mathbb{G}$-adapted processes. It can be proved using Proposition \ref{p3.3}.

\begin{prop}\label{p5.1}
Assume that $\xi$ is adapted with respect to some filtration $\mathbb{G}$ and  that $\mathcal{A}$ satisfies the hypothesis $\mathcal{D}$ with respect to $\mathbb{G}.$
If a process $\pi$ in $\cal{A}$ is optimal for $\theta\mapsto \mathbb{E}\pqa F(L^\theta) \pqc,$ then the process 
$
\xi-\int_0^\cdot \beta_t d \pqa \xi \pqc_t
$
is a $\mathbb{G}$-martingale under $P,$
where   $\beta=\pi+\frac{1}{p^\pi}\frac{d\pqa p^\pi, \xi\pqc}{d\pqa \xi,\xi\pqc},$ and $p^\pi=\mathbb{E}\pqa \frac{dP}{dQ^\pi}\left|\right. \cal{G}_\cdot \pqc.$
If $F$ is concave, then the converse holds. 
\end{prop}
\begin{proof} 
Thanks to point 2. of Assumption \ref{a0}, for every $\intervala,$
  the random variable $\xi_t-\int_0^t \pi_s d\pqa\xi \pqc_s$ is in $L^1\pta \Omega \ptc$  and so in $L^1\pta \Omega,Q^\pi\ptc$ being $\frac{dQ^\pi}{dP}$ bounded.  Then Proposition \ref{p3.3} applies to state that $\xi-\int_0^\cdot \pi_t d \pqa \xi\pqc_t$ is a $\mathbb{G}$-martingale under $Q^\pi.$  
Using Meyer Girsanov theorem,
i.e. Theorem 35, chapter III, of \cite{Protter},
 we get the necessity condition.  As far as the converse is concerned, we observe that, thanks to the hypotheses on $\mathcal{A},$ if $\xi-\int_0^\cdot \pi_t d \pqa \xi\pqc_t
$ is a $\mathbb{G}$-martingale, then for every $\theta$ in $\mathcal{A},$ the process $\int_0^\cdot \theta_t d^-\pta \xi_t- \int_0^t \pi_s d \pqa \xi\pqc_s\ptc$ is a $\mathbb{G}$-martingale starting at zero with zero expectation. This concludes the proof. 
\end{proof}

\begin{prop}\label{c3.18}
Suppose that there exists a measurable process 
$(\gamma_t,\interval)$ such that the process 
$
\xi-\int_0^\cdot \gamma_t
  d\pqa \xi \pqc_t
$
is an $\cal{A}$-martingale. 
\begin{enumerate}
\item
 If $\gamma$ belongs to $\cal{A}$ then $\gamma$ is optimal for 
$
\theta\mapsto \mathbb{E}\pqa L^\theta \pqc$.
\item Assume, furthermore, the existence of a sequence of processes
 $\pta \theta^n\ptc_{n\in \mathbb{N}}\subset \cal A$ with 
$$
\lim_{n\rightarrow +\infty}\mathbb{E}\pqa \int_0^1 \vaa \theta^n_t-\gamma_t\vac^2d\pqa \xi\pqc_t\pqc=0. $$
 If there exists an optimal  process 
$\pi,$ then $d  \pqa \xi \pqc  \pga t\in[0,1), \gamma_t\neq \pi_t \pgc=0,$ almost surely.
\end{enumerate}
\end{prop}
\begin{proof}
\begin{enumerate}
\item
The identity function $F(\omega,x) = x$ is of course strictly 
increasing and concave.
The first point is an obvious 
consequence of Proposition \ref{p7.8}.
\item Again by Proposition \ref{p7.8} and additivity,
 we deduce that a
process $\pi$ is optimal 
$\theta \mapsto E(L^\theta)$
 if and only if the process
$
\int_0^\cdot (\gamma_t-\pi_t) d\pqa \xi\pqc_t
$
is an  $\cal{A}$-martingale under $P.$
Consequently 
 $\pi$ is optimal if and only if for every
 $\theta$ is in $\cal{A}$ it holds: 
$\mathbb{E}\pqa \int_0^1 \theta_t (\gamma_t-\pi_t)d\pqa \xi\pqc_t\pqc=0.
$
In other words $\pi$ is optimal if and only if
$\gamma - \pi$ belongs to the orthogonal of $\sha$
with respect to  the Hilbert space $\shh$ of measurable processes
$R:[0,T] \times \Omega \rightarrow \R$ 
equipped with the inner product
 $\langle \theta,\ell \rangle=\mathbb{E}\left[\int_{0}^1
 \theta_{s}\ell_{s}d[\xi]_{s} \right]$.
 By the assumption of item 2. it follows that 
$\gamma$ and therefore $\gamma -\pi$
belongs to the closure of $\sha$ onto $\shh$.
Finally $\gamma - \pi$ has to vanish.
\end{enumerate}
\end{proof}


\section{The market model}
We consider a market offering two investing possibilities in the time interval $[0,1].$ Prices of the two traded assets follow the evolution of two stochastic processes $\pta
S_t^0,\interval\ptc$ and $\pta S_t,\interval\ptc.$ We could assume that
$$
S^0_t=\pta \exp(V_t),\interval\ptc,
$$
where $\pta V_t, 0\leq t\leq 1
\ptc$ is a 
positive process starting at zero with bounded variation,  
and $S$ is a continuous strictly positive process, with finite quadratic variation.

\begin{rema} 
\begin{enumerate}
\item
If  $V=\int_0^\cdot r_s ds,$ being $\pta r_t, \interval\ptc$ the short interest rate, $S^0$ represents the price process of the so called \textit{money market account}. 
Here we do not need to assume that $V$ is a riskless asset, being that assumption not necessary to develop our calculus. We only need to suppose that $S^0$ is \textit{less risky} then $S.$ 

\item Assuming that $S$ has a finite quadratic variation is not restrictive at least for  two reasons. 

Consider a market model involving an \textit{inside} trader: that means an investor having additional informations with respect to the \textit{honest} agent. Let $\mathbb{F}$ and $\mathbb{G}$ be the filtrations representing the information flow of the honest and the inside investor, respectively. Then it could be worthwhile to demand  the absence of  {\it{free lunches with vanishing risk}} (FLVR) among all simple $\mathbb{F}$-predictable strategies. Under the hypothesis of absence of (FLVR), by theorem 7.2, page 504 of \cite{DS}, $S$ is a semimartingale on the underlying probability space $(\Omega, P, \mathbb{F}).$  On the other hand $S$ could fail to be a $\mathbb{G}$-semimartingale, since (FLVR) possibly exist for the insider. Nevertheless, the inside investor is still allowed to suppose that $S$ has finite quadratic variation thanks to Proposition \ref{p3.4}.

Secondly, as already specified in the introduction, if we want to include $S$  as a \textit{self-financing}-portfolio, we have to require that $\int_0^\cdot Sd^-S$ exists. This is equivalent to assume that $S$ has finite quadratic variation, see Proposition 4.1 of \cite{RV2}.
\end{enumerate}
\end{rema}
\subsection{Portfolio strategies}
We assume the point of view of an investor whose flow of information
is modeled by a filtration \filtrationg$=\pta
\mathcal{G}_t\ptc_{t\in[0,1]}$ of $\cal F,$ which satisfies the usual assumptions.

We denote with $C^-_{b}([0,1))$ the set of processes  which have paths being left continuous and bounded on each compact set of $[0,1).$

\begin{defi}\label{d4.4}
A \textbf{portfolio strategy} is a couple of $\mathbb{G}$-adapted processes $\phi=\pta\pta h^0_t,h_t\ptc, 0\leq t<1\ptc.$ 
The market value $X$ of the portfolio strategy $\phi$ is the so called \textbf{wealth process}
$
X=h^0S^0+hS.
$
\end{defi}

We stress that there is no point in defining the portfolio strategy at the end of the trading period, that is for $t=1.$  Indeed, at time $1,$ the  agent has to liquidate his portfolio.

\begin{defi}\label{d5.3}
A portfolio strategy $\phi=\pta h^0,h\ptc$ 
is \textbf{self-financing} if  both $h^0$ and $h$ belong to $C^-_{b}([0,1)),$ the process $h$ is locally $S$-forward integrable and its wealth process $X$ verifies 
\beqn \label{sfc}
X=X_0+\int_0^\cdot  h^0_tdS^0_t + \int_0^\cdot h_t d^-S_t.
\eeqn
\end{defi}

\begin{rema}\label{R59}
When $S$ is a $\mathbb{G}$-semimartingale, 
if $h \in C^-_{b}([0,1))$  is 
locally $S$-forward integrable and previous forward integral 
coincide with classical It\^o integrals, see Proposition \ref{p3.4}.
\end{rema}

The interpretation of the first two items in definition
\ref{d5.3} is straightforward: $h^0$ and $h$ represent, respectively, the number of shares of $S^0$ and $S$ held in the portfolio; $X$ is its market value. The self-financing condition (\ref{sfc}) seems to be an appropriate formalization of the intuitive idea of trading strategy not involving exogenous sources of money. Among its justifications we can include the following ones.

As already explained in the introduction, the discrete time version of condition  (\ref{sfc}) reads as the classical self-financing condition. 
Furthermore, if
    $S$ is a $\mathbb{G}$-semimartingale, forward integrals of $\mathbb{G}$-adapted processes with left continuous and bounded
 paths, agree with  classical It\^o integrals, see Proposition \ref{p3.6} and \ref{p3.4}.

It is natural to choose as  as \textit{num\'eraire} the positive process $S^0.$ That means that prices will be expressed in terms of $S^0.$  We could
 denote with $\widetilde{Y}$
the  value of a stochastic process $(Y_t,\interval)$ \textit{discounted} with respect to $S^0:$
$
\widetilde{Y_t}={Y_t}({S^0_t})^{-1},$ for every $\interval.$

The following lemma shows that, as well as in a semimartingale model, a portfolio strategy which is self-financing is uniquely determined by its initial value and the process representing the number of shares of $S$ held in the portfolio.
We remark that previous definitions and considerations can be made without supposing that the investor is able to observe prices of $S$ and $S^0.$ 
However, we need to make this hypothesis for the following characterization of self-financing portfolio strategies.

\begin{assu}		\label{a5.4}
From now on we suppose that $S $ and $S^0$ are  $\mathbb{G}$-adapted processes.
 \end{assu}

\begin{rema}\label{R54}
Indeed, for simplicity of the formulation, we will suppose 
in most of the proofs in the sequel that $V \equiv 0$
so that $S^0 \equiv 1$.
Usual rules of calculus via regularization allow to prove statements
to the case of general $S^0$. In that case the role of the wealth process (resp. the stock price) $X$ (resp. $S$) will be replaced by $\tilde X$ (resp. 
$\tilde S$).
With our simplifying convention we will wave $X = \tilde X$, $S = \tilde S$.

\begin{prop}\label{l5.5}
 Let $\pta h_t,\intervala\ptc$ be a $\mathbb{G}$-adapted  process in
 $C^-_{b}([0,1)),$ which is locally $S$-forward integrable, and  $X_0$
 be a  $\mathcal{G}_0$-random variable. Suppose $V \equiv 0$. Then
the couple $$\phi = \pta h^0_t,
h_t, \intervala \ptc,$$ 
where $h^0_t = X_t-h_t S_t$,
 $X$ defined as 
\beqn \label{sfc1}
X=X_0 + \int_0^\cdot h_t d^-S_t,
\eeqn
 is a self-financing portfolio strategy with wealth process $X.$
\end{prop}
\begin{proof}
Let $h,$ $X_0$ and $X$ be as in the second part of the statement.
 It is clear that \\
 $h^0=\pta \pta X_t-h_tS_t\ptc,\intervala\ptc$
 is $\mathbb{G}$-adapted and belongs to $C^-_{b}([0,1)).$ 
By construction, the wealth process corresponding to the strategy $\phi=(h^0,h)$ is equal to $X.$ The conclusion follows by \eqref{sfc}.
\end{proof}
%

Proposition \ref{l5.5} leads to conceive the following definition.
\begin{defi} \label{d5.4}
\begin{enumerate}
\item
A \textbf{self-financing portfolio} is a couple $\pta X_0, h\ptc$ of a
$\mathcal{G}_0$-measurable random variable $X_0$, and  a process $h$ in $C_b^-([0,1))$ which is $\mathbb{G}$-adapted and locally $S$-forward integrable. 
\item In the sequel we let us employ the term \textit{\textbf{portfolio}}
 to denote the process $h$ (in a self-financing portfolio),
 representing the number of shares of $S$ held.
 Without further specifications the initial wealth of an investor
 will be assumed to be equal to zero. 
\end{enumerate}
\end{defi}

   


Some conditions to insure the existence of \textit{chain-rule} formulae,
 when the semimartingale property of the integrator process fails to
 hold, can be found in  \cite{FR}.
\end{rema}

\begin{assu}\label{a5.9} 
We assume the existence of a real linear space of portfolios ${\bf{\cal A}}$, that is of $\mathbb{G}$-adapted processes $h$ belonging to $C^-_b([0,1)),$ which are locally $S$-forward integrable. The set $\cal{A}$ will represent the set of all \textbf{admissible strategies} for the investor.
\end{assu}

We proceed furnishing examples of sets behaving as the set $\cal{A}$ in
Assumption \ref{a5.9}. 

\subsection{About some classes of admissible strategies} \label{ss4.2}

The aim of this section is to provide some classes of 
mathematically rigorous admissible strategies. 
We will leave most of  technical justifications to the reader;
 they are based on calculus via regularization,
see \cite{RV05} for a recent survey.

\subsubsection{Admissible strategies via It\^o fields} \label{s5.1.1}

Adapting arguments developed in \cite{FR}, we consider the following framework.
Given a $\mathbb{G}$-adapted process $(\xi_t)$
we denote by $\shc^1_\xi(\mathbb{G})$ the class of processes of the form
$H(t,\xi_t), 0 \le t \le 1)$ where 
$H(t,x), 0 \le t \le 1, x \in {\mathbb R}$ is a random field of the form
\beqn \label{IM-field}
H(t,x)=f(x)+\sum_{i=1}^n \int_{0}^t a^i(s,x)dN_s^i, \quad \interval, 
\eeqn
where
$f: \Omega\times \mathbb{R}\rightarrow \mathbb{R}$ belongs to $C^1(\mathbb{R})$ almost surely and it is $\mathbb{G}^0$-measurable for every $x,$
$H$ and $a^i:[0,1]\times \mathbb{R}\times \Omega \rightarrow \mathbb{R},\  i=1,...,n$ are $\mathbb{G}$-adapted  for every $x,$ almost surely continuous with their partial derivatives with respect to $x$ in $(t,x)$  and 
 it holds
$$
\partial_{x} H(t,x)=\partial_x f(x)+\sum_{i=1}^n\int_{0}^t \partial_{x}
 a^i(s,x)dN_s^i, \quad \interval.
$$

The following Proposition can be proved using the machinery developed in 
\cite{FR}

\begin{prop}\label{p5.11}
Let $\cal{A}$ be the  set of processes  $(h_t, \intervala)$
  such that for every $\intervala$ the process in $hI_{[0,t]}$ 
belongs to $\mathcal{C}^1_S(\mathbb{G}).$
 Then $\cal{A}$ is a real linear space satisfying the
 hypotheses of Assumption \ref{a5.9}.
\end{prop}

\subsubsection{Admissible strategies via Malliavin calculus}\label{5.25}

Malliavin calculus represents a very efficient way to introduce
a class of admissible strategies 
if the logarithm of the underlying price is a Gaussian non-semimartingale
or if anticipative strategies are admitted.
Basic notations and definitions
concerning Malliavin calculus
can be found for instance in
 \cite{NP} and \cite{N}.

We suppose that $\pta \Omega, \mathbb{F}, \mathcal{F},P\ptc$ is the
canonical probability space, meaning that $\Omega=C\pta
[0,1],\mathbb{R}\ptc$, $P$ is the Wiener measure, $W$ is the Wiener
process, $\mathbb{F}$ is the filtration generated by $W$ and the
$P$-null sets and $\mathcal{F}$ is the completion of the Borel
$\sigma$-algebra with respect to $P.$

For $p > 1, k\in \mathbb{N}^*$, $\mathbb{D}^{k,p}$ will denote the classical
Wiener-Sobolev spaces.

For any $p\geq 2,$ $L^{1,p}$ denotes the space of all functions $u$ in $L^p\pta \Omega \times [0,1] \ptc$ such that $u_t$ belongs to $\mathbb{D}^{1,p}$ for every  $\interval$  and there exists a measurable version of $\pta D_s u_t, 0\leq s,t \leq 1\ptc$ with
$
\int_0^{1} \mathbb{E}\pqa \vaa \vaa Du_t\vac\vac_{L^{2}([0,1])}^p \pqc dt<\infty.
$
The Skorohod integral $\delta$ is the adjoint of the derivative operator $D;$ its domain  is denoted by $Dom\delta.$ An element $u$ belonging to $Dom\delta$ is said Skorohod integrable. We recall that $\mathbb{D}^{1,2}$ is dense in $L^2(\Omega),$ $L^{1,2}\subset Dom\delta,$ and that if $u$ belongs to $L^{1,2}$ then, for each $\interval,$ $uI_{[0,t]}$ is still in  $L^{1,2}.$ In particular it is Skorohod integrable. We will use the notation
$
\delta\pta uI_{[0,t]}\ptc=\int_0^t u_s \delta W_s,
$
for each  $u$ in $L^{1,2}.$ The process $\pta \int_0^t u_s \delta W_s, \interval \ptc$
is mean square continuous and then it admits a continuous version, which
will be still denoted by $\int_0^\cdot u_t \delta W_t.$

\begin{defi}\label{d3.10} For every $p\geq 2,$ $L_{-}^{{1,p}}$ will be the space of all processes $u$ belonging to  $L^{1,p}$ such that
$
\lim_{\ep \rightarrow 0}D_tu_{t-\ep}
$
exists in $L^{p}(\Omega \times [0,1])$. The limiting process will be
denoted by $\pta D^-_tu_t, \interval \ptc.$
\end{defi}

Techniques similar to those of \cite{N,NP} allow to prove the
following.

\begin{prop} \label{p4.15} Let $u=\pta u^1,\dots, u^n \ptc,$ $n>1,$ be a vector of left continuous processes with bounded paths and in $L^{1,p}_{-},$ with $p>4.$ Let $v$ be a process in $L_{-}^{1,2}$ with left continuous paths such that the random variable
$
\vaa v_t\vac+\sup_{s\in[0,1]} \vaa D_sv_t\vac
$ 
is bounded.
Then for every $\psi$ in $C^1(\mathbb{R}^n)$
$\psi(u)v$ and $v$ are forward integrable with respect to $W.$ Furthermore $\psi(u)$ is forward integrable with respect to $\int_0^\cdot v_td^-W_t$ and
\beqn \nonu
\int_0^\cdot \psi(u_t)d^- \pta \int_0^t v_s d^- W_s \ptc &=&\int_0^\cdot \psi(u_t) v_t d^-W_t.
\eeqn
\end{prop}

Regarding the price of $S$ we make the following assumption.
\begin{assu}
We suppose that
$
S=S_0\exp \pta \int_0^\cdot \sigma_t dW_t+\int_0^\cdot\pta \mu_t-\frac{1}{2}\sigma_t^2\ptc dt\ptc,$ 
where $\mu$ and $\sigma$ are $\mathbb{F}$-adapted, $\mu$ belongs to $L^{1,q}$ for some $q>4,$ $\sigma$ has bounded and left continuous paths, it  belongs 
$L^{1,2}_-\cap L^{2,2}$ and the random variable 
\beqn \nonu
\sup_{t\in [0,1]} \pta\vaa \sigma_t \vac+\sup_{s\in[0,1]} \vaa D_s \sigma_t\vac 
\sup_{s,u\in[0,1]} \vaa D_s D_u\sigma_t\vac \ptc \eeqn 
is bounded.
\end{assu}

\begin{rema}\label{r5.15}
By Remark of page 32, section 1.2 of \cite{N} 
 $\sigma$ is in $L^{1,2}_-$ and $D^-\sigma=0$.
\end{rema}

Performing usual technicalities as in \cite{N, NP} it is possible
to prove that the process 
$
\log\pta S\ptc
$ 
belongs to $L^{1,q}_-.$

\begin{prop} \label{p5.16} Let $\mathcal{A}$ be the set of all $\mathbb{G}$-adapted processes $h$ in $C^-_b([0,1)),$ such that for every $0\leq t<1,$ the process $hI_{[0,t]}$ belongs to $L^{1,p}_{-},$ for some $p>4.$  
Then $\cal{A}$ is a real linear space satisfying the hypotheses of Assumption \ref{a5.9}.
\end{prop}

\begin{proof}
Let $h$ be in $\cal{A}.$ We set 
$
A=\log(S)-\log(S_0)+\frac{1}{2}\int_0^\cdot\sigma_t^2dt=\int_0^\cdot \sigma_tdW_t+\int_0^\cdot \mu_tdt.
$
We recall that, thanks to Proposition \ref{l2.1}, for every $0\leq t <1,$ $hI_{[0,t]}$ is $S$-forward integrable if and only if $hI_{[0,t]}S$ is forward integrable with respect to $A.$ Let $0\le t<1,$ be fixed.
Each component of the vector process $u=\pta hI_{[0,t]},\log(S)\ptc$ belongs to $L^{1,p}_-$ for some $p>4$ and it has left continuous and bounded paths. We can thus apply Proposition \ref{p4.15} to state that $hI_{[0,t]}S$ is forward integrable with respect to $\int_0^\cdot \sigma_t dW_t.$ 
This implies that $hI_{[0,t]}S$ is $A$-forward integrable. Letting $t$ vary in $[0,1)$
we find that $h$ is $S$-improperly integrable and  we conclude
  the proof. \end{proof}

\subsubsection{Admissible strategies via substitution} \label{e5.21}

Let $\mathbb{F}=\pta \mathcal{F}_t\ptc_{t\in[0,1]}$ be a filtration on $\pta \Omega, \mathcal{F}, P\ptc,$ with $\mathcal{F}_1=\mathcal{F},$ and $G$ an $\mathcal{F}$ measurable random variable with values in $\mathbb{R}^d$.
We set $
\mathcal{G}_t=\pta \mathcal{F}_{t} \vee \sigma(G)\ptc,
$ and we suppose that $\mathbb{G}$ is right continuous:
$$
\mathcal{G}_t=\bigcap_{\ep > 0}\pta \mathcal{F}_{t+\ep} \vee \sigma(G)\ptc.
$$

In this section $\mathcal{P}^{\mathbb{F}}$ ($\mathcal{P}^{\mathbb{G}},$ resp.) will denote the $\sigma$-algebra of $\mathbb{F}$ (of $\mathbb{G},$ resp.)-predictable processes.
$E$ will be the Banach space of all continuous functions on $[0,1]$ equipped with the uniform
norm $\left|\left|f\right|\right|_{E}=\sup_{t\in[0,1]}\left|f(t)\right|.$

\begin{defi}An increasing sequence of random times $\pta
T_k\ptc_{k\in\mathbb{N}}$ is said \textit{\textbf{suitable}} if \\
$P\pta \cup_{k=0}^{+\infty}\pga
T_k=1 \pgc \ptc =1.$
\end{defi}

Let $\sha^{p,\gamma}(G)$ be the set of processes $(u_t)$ where $u_t = h(t,G)$ 
where $h(t,x)$ is a random field fulfilling the following Kolmogorov
type conditions:
there is a suitable sequence of stopping times $(\sht_k)$ for which
 $$
\mathbb{E}\pqa \sup_{t\in[0, \sht_k]}\left|h(t,x)-h(t,y)\right|^{p}\pqc
\leq c \left|x-y\right|^{\gamma}, \quad \forall x, y \in C.
$$

 We assume that $S$ and $S^0$ are $\mathbb{F}$-adapted, and that
 $S$ is an   $\mathbb{F}$-semimartingale.

We observe that this situation arises when the investor trades as an \textit{\textbf{insider}}, that is having an extra information about prices, at time $0,$ represented by the random variable $G.$
   
Performing substitution formulae  as in \cite{RV, RV2, RV3, ER1},
 it is possible to establish the following result.

\begin{prop}
Let $\sha$ be the set of processes $h$ such that, for every $0\leq t
<1,$ the process $hI_{[0,t]}$ belongs to $\mathcal{A}^{p,\gamma}
$ for some $p>1$ and $\gamma > 0$. 
 Then $\cal{A}$ satisfies the hypotheses of  Assumption \ref{a5.9}.  
\end{prop}

\subsection{Completeness and arbitrage: $\mathcal{A}$-martingale measures}

\begin{defi}
Let $h$ be a self financing portfolio in $\cal{A}$  which is $S$-improperly forward integrable and $X$ is its wealth process. 
Then $h$ is an \textbf{arbitrage} if  $X_1=\lim_{t\rightarrow 1}X_t$ exists almost surely, $P(\pga X_1 \geq 0\pgc )=1$ and  $P(\pga X_1>0\pgc)>0.$ 
\end{defi} 

\begin{defi}
We say that the market is \textbf{$\mathcal{A}$-arbitrage free} 
if no self-financing strategy $h$ in $\mathcal{A}$ is an arbitrage.
\end{defi}

\begin{defi}
A probability measure $Q \sim  P$ is said \textbf{$\mathcal{A}$-martingale measure} if under $Q$ the process $S$ is an $\mathcal{A}$-martingale according to definition \ref{d3.1}.
\end{defi}


For the following Proposition the reader should keep in mind the notation in equality (\ref{AX}). We omit its proof which is a direct application of Corollary \ref{c4.16}.

\begin{prop} \label{p5.27}
Let  $\mathcal{A}=\mathcal{A}_S.$  Suppose that
 $d\pqa S\pqc_t=\sigma(t,S_t)^2 S_t^2dt,$ where $\sigma$ satisfies Assumption 
\ref{a4.15}.  
 Moreover we suppose that the unique solution of equation 
\eqref{e11} admits a density for $0 < t \le 1$.
If there exists a $\cal A$-martingale measure then  the law of $S_t$ 
is absolutely continuous with respect to Lebesgue measure,
 for every $0 < t \le 1$.
\end{prop}

\begin{prop} \label{p5.28}
If there exists an $\mathcal{A}$-martingale measure $Q,$ the market is $\mathcal{A}$-arbitrage free.
\end{prop}
\begin{proof}
Suppose again that $V \equiv 0$ and that $h$ is an $\cal{A}$-arbitrage.
 Since $S$ is an $\mathcal{A}$-martingale under $Q,$ 
 we find $\mathbb{E}^Q[X_1]=\mathbb{E}^Q[\int_0^1 h_td^-S_t]=0$. This contradicts the arbitrage condition $Q(\pga X_1>0\pgc)>0.$
\end{proof}



We proceed discussing \textit{completeness}.
\begin{defi}
A  \textbf{contingent claim} $C$ is an $\cal{F}$-measurable random variable.
We denote $\tilde C = \frac{C}{S^0_T}$.
 $\mathcal{L}$ will be a set of $\cal{F}$-measurable random variables; it will represent all the {contingent claims} the investor is interested in.
\end{defi}

\begin{defi} 
\begin{enumerate}
\item A contingent claim $C$ is said \textbf{$\mathcal{A}$-attainable} if there exists a self financing portfolio $(X_0,h)$ with $h$
in $\mathcal{A},$ which is $S$-improperly forward integrable,  
such that the corresponding wealth process $X$ verifies
$\lim_{t\rightarrow 1}X_t=C,$ almost surely. The portfolio $h$ is said the \textbf{replicating} or \textbf{hedging} portfolio for $C,$ $X_0$ is said the \textbf{replication price} for $C.$
\item The market is said to be $\pta \mathcal{A},\mathcal{L}\ptc$-{\bf
 attainable} if every contingent claim in $\cal{L}$ is attainable 
trough a portfolio in $\mathcal{A}.$
\end{enumerate}
\end{defi}

\begin{assu}\label{A1}
For every $\mathcal{G}_0$-measurable random variable $\eta,$ and $(h_t)$ 
in $\cal{A}$ the process $u=h\eta,$ belongs to $\cal{A}.$ 
\end{assu}

\begin{prop}\label{p5.33}
Suppose that the market is $\cal{A}$-arbitrage free, and that Assumption $\ref{A1}$ is realized. Then the replication price of an attainable contingent claim is unique. 
\end{prop}
\begin{proof}
Let $(X_0,h)$ and $(Y_0,k)$ be two replicating portfolios for a contingent claim $C,$ with $h$ and $k$ in $\cal{A},$ and wealth processes $X$ and $Y$, respectively. We have to prove that $$
P\pta \pga X_0-Y_0\neq 0\pgc \ptc=0.
$$ 
Suppose, for instance, that $P\pta X_0-Y_0>0\ptc \neq 0.$ We set $A=\pga  X_0-Y_0>0\pgc.$ By Assumption \ref{A1}, $I_{A}(k-h)$ is a portfolio in $\cal{A}$ with wealth process  $I_A (Y_t-X_t).$ Since both $(X_0,h)$ and $(Y_0,k)$ replicate $C,$ $\lim_{t\rightarrow 1} I_{A}(Y_t-X_t)=I_{A}(X_0-Y_0),$ with $P(\pga I_A(X_0-Y_0>0)\pgc )>0.$ Then  $I_A(k-h)$ is an $\cal{A}$-arbitrage and this contradicts the hypothesis.   
\end{proof}

\begin{prop}\label{p5.30}  Suppose that there
  exists an $\mathcal{A}$-martingale measure $Q$.  Then the following
  statements are true.

\begin{enumerate}
\item Under Assumption
 \ref{A1}, the replication price of an $\cal{A}$-attainable contingent
 claim $C$ is unique and equal to $\mathbb{E}^Q\pqa \widetilde C\left|\right.\mathcal{G}_0\pqc.$
\item Let $\mathcal{G}_0$ be trivial. If  $Q$ and $Q_1$ are two
  $\mathcal{A}$-martingale measures, then 
$\mathbb{E}^Q[ \widetilde{C}]=\mathbb{E}^{Q_1}[\widetilde C],$ for every $\cal{A}$-attainable contingent claim $C$. In particular, if  the market is $(\cal{A},\cal{L})$-attainable and $\mathcal{L}$ is an algebra, all $\mathcal{A}$-martingale measures coincide on the $\sigma$-algebra generated by all bounded discounted contingent claims in $\mathcal{L}.$ 
\end{enumerate}
\end{prop}
\begin{proof}
Suppose again$V \equiv 0$.
Let $(X_0,h)$ be a replicating $\cal{A}$-portfolio for $C$.
Then $$
\mathbb{E}^{Q}\pqa C\left|\right. \mathcal{G}_0\pqc=X_0+\mathbb{E}^{Q}\pqa \int_0^1 h_td^-S_t \left|\right.\mathcal{G}_0\pqc.
$$
We observe that 
$
\mathbb{E}^{Q}\pqa \int_0^1 h_td^-S_t \left|\right.\mathcal{G}_0\pqc=0.
$
In fact, if $\eta$ is  a $\mathcal{G}_0$-measurable random variable, then, thanks to  Assumption  \ref{A1}, $\eta h$ belongs to $\mathcal{A},$ so as to have
$\mathbb{E}^Q\pqa \pta \int_0^1 h_td^-S_t\ptc \eta \pqc=\mathbb{E}^Q\pqa  \int_0^1 \eta h_td^-S_t\pqc=0.$  This implies point 1.

If $\mathcal{G}_0$ is trivial, we deduce that, if $Q$ and $Q_1$ are
 two $\mathcal{A}$-martingale measures,  $\mathbb{E}^{Q}[
 C]=\mathbb{E}^{Q_1}[C],$ for every
 $\mathcal{A}$-attainable contingent claim. The
  last point is a consequence  of
the monotone class theorem, see theorem 8, chapter 1 of \cite{Protter}. 
\end{proof}

\subsection{Hedging}
In this part of the paper we price contingent claims via partial
differential equations. In particular, within a non-semimartingale model,
 we emphasize robustness of
Black-Scholes formula for \textit{European}, \textit{Asian} and some 
path dependent contingent claims depending on a finite number 
of dates of the underlying price.

We suppose here that $d\pqa S \pqc_t=\sigma^2(t,S_t)S_t^2dt$ and
$dV_t=rdt,$
 with $ r>0$ and $\sigma:[0,1]\times (0,+\infty)\rightarrow \mathbb{R}.$   
We suppose the existence of constants $c_1,c_2$ such that $0< c_1 \leq \sigma \leq c_2$.

Similar results were obtained by \cite{KS} and \cite{Z}.
Examples of non-semimartingale processes $S$ of that type
 can be easily constructed.
They are related to processes $X$ such that $[X] = {\rm const} \ t$.
A typical example is a Dirichlet process which can be written as
 Brownian motion plus a 
zero quadratic variation term. A not so well-known example is 
given by bifractional Brownian motion $X = B^{H,K}$ for indices
$H \in ]0,1[, K \in ]0,1]$ such that $HK = \frac{1}{2}$,
see for instance \cite{RT}. This process is neither
a semimartingale nor a Dirichlet process.


\begin{prop} \label{p5.31}
Let $\psi$ be a function in $C^0(\mathbb{R})$. Suppose that there exists $\pta v(t,x), \interval, x\in \mathbb{R}\ptc$ of class $C^{1,2}([0,1)\times \mathbb{R})\cap C^0([0,1]\times \mathbb{R}),$ which is a  solution of the following Cauchy problem
\beqn \label{e14}
\left\{
\begin{array}{lll}
\partial_tv(t,y)+\frac{1}{2}(\widetilde{\sigma}(t,y))^2y^2\partial_{yy}^{(2)}v(t,y)&=&0 \quad \mbox{ on }  [0,1)\times{\mathbb{R}}
\\
v(1,y)&=&\widetilde{\psi}(y),
\end{array}
\right.
\eeqn
where 
$$
\left\{
\begin{array}{ll}
\widetilde{\sigma}(t,y)=\sigma(t,ye^{rt})&\quad \forall (t,y) \in [0,1]\times\mathbb{R},
\\ \widetilde{\psi}(y)=\psi(ye^r)e^{-r}&\quad \forall y \in \mathbb{R}.  
\end{array}
\right. 
$$
Set 
$$
h_t=\partial_yv(t,\widetilde{S}_t),\quad \intervala,  \quad X_0=v(0,S_0).
$$
Then $(X_0,h)$ is a self-financing portfolio replicating the contingent claim $\psi(S_1).$
\end{prop}

\begin{proof} 
Again, for simplicity, we consider the case $r = 0$.
Assumption \ref{a5.4} tells us that $h$ is a $\mathbb{G}$-adapted process in $C^-_b([0,1)).$  By Proposition \ref{p2.1}, $h$ is locally $S$-forward integrable. 
Applying  Proposition \ref{p2.1}, recalling equation (\ref{e14}), equalities (\ref{sfc}) 
 we find that  
$$
X_t=v(t,S_t), \quad \forall \intervala.
$$
In particular $X_0+\lim_{t\rightarrow 1} \int_0^t h_sd^-S_s$ exists
finite and coincides with $v(1,S_1)= \psi(S_1).$
\end{proof}

\begin{rema}
In particular, under some minimal regularity assumptions on  $\sigma$ and  no degeneracy, the market  is $(\mathcal{A}_S,\mathcal{L})$-attainable, if $\mathcal{L}$ equals the set of all contingent claims of type $\psi(S_1)$ with $\psi$ in $C^0(\mathbb{R})$ with linear growth.
\end{rema}
Enlarging suitably ${\mathcal A}$ 
and solving successively and recursively equations of the type
\eqref{e14}, it is possible to replicate contingent claims 
of the type $C = \psi(X_{t_1},\cdots, X_{t_n})$ with $0 \le 
t_1 <\cdots < t_n = 1$ and $\psi: {\mathbb R}^n \rightarrow  {\mathbb R}$
continuous with polynomial growth.

The proposition below provides a suitable framework for this.

\begin{prop}			\label{EDPtratti}
Let $r = 0$ so $V \equiv 0$.
Suppose $d[S]_{t}=\sigma^{2}(t,S_{t})S^{2}_{t}dt$ and $\psi$ a function in $C^{0}(\R^{n})$ with polynomial growth. 
Let $0=t_{0}< t_{1}<\ldots < t_{n}=1$, $n\geq 2$. 
Suppose that there exist functions 
$v^{1},\ldots, v^{n}$ such that 
\begin{itemize}
\item $v^{i}\in C^{1,2}([t_{i-1},t_{i})\times \R^{i})\cap C^{0}([t_{i-1},t_{i}]\times \R^{i}])$, $1\leq i\leq n$; 
\item and denoting shortly 
$v^{i}(t,y):=v^{i}(t,y_{1},\ldots, y_{i-1},y)$ for $1\leq i\leq n$ we have
\be			\label{EQedp} 
\left\{
\ba{ll}
\partial_{t}v^{n}(t,y)+\frac{1}{2}\sigma^2(t,y)y^{2}\partial^{(2)}_{yy}v^{n}(t,y)=0  & \textrm{ on } [t_{n-1},1) \times \R\\
v^{n}(1,y_{1},\ldots, y_{n-1},y)=\psi(y_{1},\ldots, y_{n-1},y) &
\ea
\right.
\ee
and for $i=1,\ldots, n-1$ 
\be			\label{EQedp2} 
\left\{
\ba{ll}
\partial_{t}v^{i}(t,y)+\frac{1}{2}\sigma^2(t,y)y^{2}\partial^{(2)}_{yy}v^{i}(t,y)=0  & \textrm{ on } [t_{i-1},t_{i}) \times \R\\
v^{i}(t_{i},y_{1},\ldots, y_{i-1},y)=v^{i+1}(t_{i} , y_{1},\ldots,y_{i-1},y,y).&
\ea
\right.
\ee
In particular $v^{1}(t_{1},y)=v^{2}(t_{1},y,y)$.
\end{itemize}
Setting 
\begin{eqnarray*}
h_{t}&=&I_{[0,t_{1}]}(t) \partial_{y}v^{1}(t,S_{t}) + 
\sum_{i=2}^{n} I_{(t_{i-1},t_{i}]} (t) \partial_{y}v^{i}(t,S_{t_{1}},\ldots, 
S_{t_{i-1}},S_{t})  \\
X_{0}&=&v^{1}(0,S_{0}) \; .
\end{eqnarray*}
Then $(X_{0},h)$ is a self-financing portfolio replicating the contingent claim $\psi(S_{t_{1}},\ldots, S_{t_{n}})$.
\end{prop}


The result of Proposition \ref{p5.31} can also be adapted to hedge Asian
 contingent claims, that is contingent claims  $C$ depending on the mean of $S$ over the traded period:
 $C=\psi\pta  \frac{1}{S_1}\pta \int_0^1 S_t dt\ptc\ptc S_1,$ for some $\psi$ in $C^0(\mathbb{R}).$  
\begin{prop} \label{PAsian}
Suppose that $\sigma(t,x)=\sigma,$ for every  $(t,x)$ in $[0,1]\times \mathbb{R},$ for some $\sigma>0.$ Let $\psi$ be a function in $C^0(\mathbb{R})$  and $v(t,y)$ a continuous solution of class $C^{1,2}([0,1)\times  \mathbb{R})\cap C^{0}([0,1]\times \mathbb{R})$ of the following Cauchy problem
\beqn \nonu
\left\{
\begin{array}{lll}
\frac{1}{2}\sigma^2 y^2 \partial_{yy}^{(2)}v(t,y)+ (1-ry)\partial_{y}v(t,y)+\partial_tv(t,y)&=&0, \quad \mbox{ on }  [0,1)\times{\mathbb{R}}
\\
v(1,y)&=&\psi(y).
\end{array}
\right.
\eeqn
Set  $Z_t=\int_0^t S_sds-K,$ 
for some $K>0,$ $X_0=v(0,\frac{K}{S_0})S_0$ and $h_t=v(t,\frac{Z_t}{S_t})-\partial_y v(t,\frac{Z_t}{S_t})\frac{Z_t}{S_t},$ for all $0\leq t \leq 1.$ Then $\pta X_0,h \ptc$ is a self-financing portfolio which replicates the contingent claim $\psi \pta \frac{1}{S_1}\pta \int_0^1 S_tdt-K\ptc\ptc S_1.$
\end{prop}

\begin{proof}
Again for simplicity we will suppose $r = 0$.
We set $\xi_t=\frac{Z_t}{S_t}, \interval$. Applying Proposition \ref{p2.1} to the function $u(t,z,s)=
v(t,\frac{z}{s})s$ and using the equation fulfilled by $v$ we can 
 expand the process  $v(t,\xi_t)S_t, \intervala)$ as  follows:
\beqn \label{f19}
u(t,Z_t,S_t)=v\pta t,\xi_t\ptc S_t=v\pta 0,\xi_0\ptc S_0+\int_0^t h_td^-S_t.
\eeqn 
By arguments which are similar to those used in the proof of Proposition \ref{p5.31}, it is possible to show that $h$ is a self-financing portfolio and that (\ref{f19}) implies that $u(t,Z_t,S_t)=X_t$ for every $\intervala.$ Therefore $\lim_{t \rightarrow 1}X_t$ is finite and equal to $\psi\pta \xi_1\ptc S_1 e^{-r}.$ This concludes the proof. 
\end{proof}


\subsection{On some 
sufficient conditions for no-arbitrage}			\label{nuova}

\subsubsection{Some illustration on  weak geometric 
Brownian motion}

Before we would like to give a first class
of non-arbitrage conditions related to the existence of
a $\sha$-martingale measure. 

For a process $X$ we define the set $\mathcal{A}^{n}_{X}$ as the space of all processes $h$ of type: 
\[
h_{t}=I_{[0,t_{1}]}(t)u^{1}(t,X_{t})+\sum_{i=2}^{n}I_{(t_{i-1},t_{i}]}(t)
u^{i}(t,X_{t_{1}},\ldots, X_{t_{i-1}},X_{t})
\]
where $0=t_{0}< t_{1}<\ldots < t_{n}=1$ and for every $i=1,\ldots, n$ 
\begin{itemize}
\item $u^{i}:
[0,1] \times \R^{i}\longrightarrow \R$ of class $C^{1}((t_{i-1},t_{i})\times \R^{i}) \cap C^{0}([t_{i-1},t_{i}]\times \R^{i})$ 
\item $u^{i}$ and its derivatives have polynomial growth on each
interval $(t_{i-1},t_{i}]$.
\end{itemize}

\begin{defi}
A continuous process $X$, is said {\bf weak $\sigma$-geometric Brownian motion 
of order} $n$ if, for every $0\leq t_{0}< t_{1}<\ldots < t_{n}\leq 1$
\[
(X_{t_{1}},\ldots, X_{t_{n}})(P)= \textrm{ law of } (Z_{t_{1}},\ldots, Z_{t_{n}})
\]
and $Z$ is a weak solution of equation $Z_{t}=X_{0}+\int_{0}^{\cdot}\sigma \, Z\, dW_{t}$
\end{defi}

\begin{rema} \label{R4.32}
\begin{enumerate} Let $n \ge 4$.
\item
With the help of Proposition   \ref{p2.1},
examples of such a process can be produced for instance
setting $X_t = \exp(\sigma B_t - \frac{\sigma^2 t}{2})$
whenever $B$ is a weak Brownian motion of order $n$. 
 \item If  $B$ is a weak Brownian motion
of order $n$ then $X$ is a 
is a finite quadratic variation process
with $[X]_t = \int_0^t \sigma^2 X_t^2 dt$.
\end{enumerate}
\end{rema}
Similar arguments as in the proof of Proposition
\ref{p}, performed in every subinterval
$(t_i, t_{i+1}]$, allow to prove the following.
\begin{prop} \label{P4.32}
Suppose that $S$ is a weak $\sigma-$geometric Brownian motion of order $n$
 with $d[S]_{t}=\sigma^{2}S^{2}_{t}dt$. Then $S$ is an $\mathcal{A}^{n}_{S}$-martingale.
\end{prop}
\begin{defi}
Let $\mathcal{L}^{n}_{S}$ be the set of all contingent claims of type $\psi(S_{t_{1}},\ldots, S_{t_{n}})$ such that the hypotheses of Proposition \ref{EDPtratti} are verified and 
the process $h$ belongs to $\mathcal{A}^{n}_{S}$.
\end{defi}
\begin{coro} \label{C4.32}
Suppose that $S$ satisfies the hypotheses of previous proposition. Then the market is $\mathcal{A}^{n}_{S}$-viable and $\mathcal{L}^{n}_{S}$-complete.
\end{coro}

\subsubsection{On some Bender-Sottinen-Valkeila type conditions}

The rest of this subsection is inspired by the work of \cite{BSV} whose results
are reformulated below in a similar but different framework. \\
For simplicity we will suppose again $V \equiv 0$ so
that the underlying is discounted.
We start with some notations and a definition.
Let $y_0 \in \R, t \in [0,1]$. We denote by
$C_{y_0}([0,1])$  the Banach space of continuous
function $\eta: [0,1] \rightarrow \R$
such that $\eta(0) = y_0$.
For $t \in [0,1] $ we define the shift operator
$\Theta_{t}: C([0,1]) \rightarrow  C([-1,0])$ defined by
$(\Theta_t \eta)(x) = \eta(x+t), \ x \in [-1,0]$.
We remind that continuous functions defined on some real
interval $I$ are naturally prolongated by continuity on 
the real line.
With a real process  $S = (S_t, t \in [0,1])$ 
we associate the ``window'' process 
 $ S_t(\cdot)$ with values in $C([-1,0])$, 
setting $ S_t(x) = S_{t+x}, x \in [-1,0]$. $S$ denotes the random element 
$S:\Omega\longrightarrow C([0,1])$, $\omega\mapsto S(\omega)$.

\begin{defi} \label{DFSC}
Let $Y = (Y_t, t \in [0,1])$ be a process such that $Y_0 = y_0$ for some 
$y_0 \in \R$. $Y$ is said to fulfill the {\bf full support condition}
 if for every $\eta \in C_{y_0}([0,1]) $ one has 
$ P \{ \Vert Y - \eta \Vert_\infty  \le \varepsilon \} > 0$.
\end{defi}
That  notion is present in the classical stochastic analysis literature, see
for instance \cite{maro}.
\cite{GRS} introduced a refined version of it which
is called the CFS (conditional full support) condition.

\begin{prop} \label{PFSC} 
\begin{enumerate}
\item Let $M$ be a local martingale such there is
a progressively measurable process $(\sigma_t, t \in [0,1])$
such that $[M]_t = \int_0^t \sigma^2_s ds, t \in [0,1]$
and a constant $c > 0$ with $\sigma_{s} \ge c$, $s\in [0,1]$. 
We will say in that case that $M$ is a non-degenerate.
Then $M$ fulfills the full support condition.
\item Let $G $ be an independent process from a process
$M$ fulfilling the full support condition. Suppose that $G_0 = 0$.
Then $X = M + G$ also fulfills the full support condition.
\item Let $f: \R \rightarrow \R$ be strictly increasing and continuous.
If $Y$ fulfills the full support condition then $f(Y)$
also fulfills the support condition.
\item Let $Y$ be a non-degenerate martingale, for instance a Brownian motion, and an independent process $\xi$. 
The process $X=e^{Y+\xi}$ fulfills the mentioned condition.
\end{enumerate}
\end{prop}
\begin{proof}
\begin{enumerate}
\item It is well known that the standard Wiener process  fulfills
the full support condition. One possible argument follows 
directly from a Freidlin-Wentsell type estimate.
Given a Brownian motion $\shw$, $\left( \shw_{ct},  t \in [0,1] \right)$
fulfills the full support condition  
by a law rescaling argument.
By Dambis, Dubins-Schwarz theorem (see Theorem 1.6 chapter V of \cite{RY}),
there is a Brownian motion $\shw$ such that $M  = \shw_{\int_0^\cdot \sigma^2_s ds}$.
Let $\eta \in C_0([0,1])$; since 
$\Vert M - \eta \Vert_\infty \ge  \Vert \shw_{c \cdot} - \eta \Vert_\infty,$
 then for any $\varepsilon > 0$,
$$ P \{ \Vert M - \eta \Vert_\infty \le \varepsilon \} 
\ge  P \{\Vert \shw_{c \cdot} - \eta \Vert_\infty  \le \varepsilon \} > 0 $$  
and the result follows.
\item Let $ g \in C_0([0,T])$ be a realization of $G$. 
We set $\Psi(g) = P\{ \Vert M + g - \eta \Vert \le  \varepsilon \}.$
Clearly $P\{ \Vert  M + G - \eta \Vert \le \varepsilon \} 
= E( \Psi(G))$.
By item 1. $\Psi (g) $ is strictly positive for any $g$,
so the result follows.
\item Obvious.
\item It follows from previous items.
\end{enumerate} \end{proof}
%



\begin{ass} \label{assA} 
\
Let $S$ be a continuous process such that $S_{0}=s_{0}$ and $\sigma:[0,1]\times C([-1,0])\longrightarrow \R$ be a continuous functional. 
Let $\mathcal{A}$ be a class of self-financing portfolios $h$ with corresponding strategies 
$\phi=\pta\pta h^0_t,h_t\ptc, 0\leq t<1\ptc$ with associated wealth
 process $X_{t}(\phi)=h^0_t+h_{t}\cdot S_{t}$. 
For every $h \in \mathcal{A}$, we suppose the existence of a continuous functional $\mathcal{H}:[0,1]\times C[-1,0]\longrightarrow \mathbb{R}$ 
with polynomial growth such that 
$h_{t}=\mathcal{H}(t,S_{t}(\cdot))=\mathcal{H}(t,\Theta_t S )$, $t\in [0,1[$.

We say that \textbf{$\mathbb{\mathcal{A}}$ fulfills Assumption \ref{assA}
 (with respect to $\sigma$)} 
if there is a continuous functional $\mathcal{V}=\mathcal{V}_{\phi}:C([0,1])\longrightarrow \R$ such that, 
whenever $[S]_{t}=\int_{0}^{t}\sigma^{2}(s,S_{s}(\cdot)) S_s^2, ds $ with respect to some probability $Q$, 
then 
\begin{equation}	\label{eq Vteta}
\mathcal{V}_{\phi}(S)=\int_{0}^{1}h_{s}d^{-}S_{s}
\qquad Q \textrm{ a.s. }
\end{equation}
In particular the right-hand side forward integral exists with respect to $Q$.
\end{ass}
We recall that $BV([0,1])$ denotes the linear space
of  bounded variation function $f: [0,1] \rightarrow \R$
equipped with the topology of weak convergence of the related
measures. 

\begin{prop}		\label{propAssA}
Let $S$ be a continuous process such that $S_{0}=s_{0}$ and $\sigma:[0,1]\times C([-1,0])\lra \R$ be continuous. 
Let $\mathcal{A}$ be constituted by the self-financing portfolios $h$
 such there exists a continuous $\varphi:[0,1]\times \R^{n} 
\times \R\lra \R$ with polynomial growth such that $\varphi\in C^{1}\pta  [0,1[ \times\R^{n}\times \R\ptc $, $(t,v_{1},\ldots, , v_{n},x)\mapsto \varphi(t,v_{1},\ldots, , v_{n},x)$, and \\
$
h_{t}=\varphi\pta  t , V^{1}_t \pta S_{t}(\cdot)\ptc , 
\ldots, V^{n}_t\pta S_{t}(\cdot)\ptc, S_{t}\ptc =\varphi\pta  t , V^{1}_t \pta \Theta_{t}S\ptc , 
\ldots, V^{n}_t\pta \Theta_{t}S \ptc, S_{t}\ptc 
$, i.e.
$$
\mathcal{H}(t,\gamma)=\varphi\pta t, V^{1}_t(\gamma),\ldots, 
V^{n}_t(\gamma),\gamma(0)\ptc $$ where $\gamma \mapsto V^{i}(\gamma)$ is continuous from $C([-1,0])$
 to the class of bounded variation functions
$BV([0,1])$. 
Then $\mathcal{A}$ fulfills Assumption \ref{assA} with respect to $\sigma$.
\end{prop}

\begin{proof}
In order to relax the notations we just suppose $n= 1$.
We set $\tilde \varphi(t,v,x) = \int_0^x \varphi(t,v,y) dy, \,  t \in \R,
v,x \in [0,1]$. Let $Q$ be a probability under which $[S]_{t}=\int_{0}^{t}\sigma^{2}(s,S_{s}(\cdot))ds$.
By It\^o formula Proposition \ref{p2.1}  applied
reversely to $\tilde \varphi (t, Y_t,S_t)$ for $Y_t:=V^1_t(S_t(\cdot))$ from $0$ to $1$, we get 
\begin{eqnarray} \label{EEEStrat}
\int_{0}^{1} h_{s}d^{-}S_{s}&=& \tilde \varphi(1, V^1_1(S_1(\cdot)),S_1) 
-  \tilde \varphi(0, V^1_0(S_0(\cdot)), S_0)  
- \int_0^1  \partial_{s} \tilde \varphi(s, V^1_s(S_s(\cdot)),S_s) ds \nonumber  \\ 
 && \\
&-& \frac{1}{2}  \int_0^1  \partial_{x} \varphi(s, V^1_s(S_s(\cdot)),S_s) \sigma^{2}(s,S_s(\cdot)) ds
-  \int_0^1  \partial_{v} \tilde \varphi(s, V^1_s(S_s(\cdot)),S_s) dV^1_s(S_s(\cdot)). \nonumber
\end{eqnarray}
Setting 
\begin{eqnarray} \label{EEEStrat1}
 \shv(\eta) &= & \tilde \varphi(1, V^1_1( \Theta_{1}\eta),\eta(1)) 
-  \tilde \varphi(0, V^1_0 (\Theta_0 \eta), \eta(0))  
- \int_0^1   \partial_{s} \tilde \varphi(s, V^1_s(\Theta_s \eta),\eta(s)) ds 
\nonumber\\
&& \\
&-&  \frac{1}{2} \int_0^1\partial_{x}    \varphi(s, V^1_s (\Theta_s \eta),\eta(s)) 
\sigma^{2}(s,S_s(\cdot)) ds
-  \int_0^1 \partial_{v} \tilde \varphi(s, \Theta_s \eta,\eta(s)) dV^1_s(\Theta_s 
\eta). \nonumber
\end{eqnarray}
The continuity of previous expression is obvious and so
Assumption \ref{assA} is fulfilled.
\end{proof}

Examples of classes of strategies which fulfill Assumption \ref{assA} by Proposition \ref{propAssA}.
\begin{exam} \label{EVTeta} 
Let  $S$ be a finite quadratic variation such that
 $S_0 = s_0$
for some $s_0 \in \R$.
We suppose moreover
$[S]_t =  \int_{0}^{t} \sigma^{2}\pta s,S_{s}(\cdot) \ptc S_{s}^2 ds$ 
where $\sigma:[0,1]\times C([-1,0])\lra \R$ is 
continuous with linear growth. 
\begin{enumerate}
\item The class of strategies are determined by $\varphi:[0,1]\times \R\lra \R$ of class $C^{1}$, $(t,x) \rightarrow \varphi(t,x)$.
We set
 $\mathcal{H}(t,\eta)=\varphi \pta t,\eta(0)\ptc$. 
We denote $\tilde{\varphi}(t,x)=\int_{0}^{x}\varphi(t,z)dz$. 
By  It\^o's formula given 
in Proposition \ref{p2.1} we get 
$$ \int_0^t \varphi(s,S_s) d^- S_s = \tilde \varphi(t,S_t) 
- \tilde \varphi(0,S_0) -  \int_0^t \partial_s \tilde{\varphi}(s,S_s) ds 
- \frac{1}{2} \int_0^t \partial_{x} \varphi (s,S_s) ds $$
Setting 
$$\shv(\eta) = \tilde \varphi(1, \eta(1)) 
- \tilde \varphi(0,\eta(0)) -  \int_0^t \partial_s 
\tilde \varphi(s,\Theta_s \eta(0)) ds 
- \frac{1}{2} \int_0^t \partial_{x} \varphi (s, \Theta_s \eta(0)) ds,
 $$
Assumption \ref{assA} is verified via Proposition \ref{propAssA}. The class here defined is a 
subclass of $\mathcal{A}_{S}$ defined in the introduction.
\item As emphasized in  \cite{BSV}, possible choices of $V^{i}$ given in Proposition \ref{propAssA}, are given by 
$V^{1}_t(\gamma)=\min_{r\in[-t,0]} \{ \gamma(r)\}$, $V^{2}_t(\gamma)=\max_{r\in[-t,0]} \{ \gamma(r)\}$, $V^{3}_t(\gamma)=\int_{-t}^{0}\gamma(r)dr$. 
According to the \cite{BSV} terminology, we could call those functional $V^{i}$, $i=1,2,3$, \emph{inside factors}.
\end{enumerate}
\end{exam}

\begin{rema}
\begin{enumerate}
\item Let $0 = t_0 <  \cdots < t_n =1$ be a subdivision
of $[0,1]$ interval. In the statement of Proposition \ref{propAssA} the class of strategies can be enlarged considering similar classes of strategies on each subinterval 
$]t_{i}, t_{i+1}]$. 

The class of strategies $\mathcal{A}$ constituted by 
the portfolio strategies $h$
 such that for every $i$ there exists an integer $n_{i}\geq 0$ such that
$$
h_{t} I_{]t_{i},t_{i+1}]}=
\varphi^{i} 
\pta t, S_{t_{1}}, \ldots, S_{t_{i}}, S_{t}, V^{1}_{t}(S_{t}(\cdot)), \ldots, V^{n_{i}}_{t}( S_{t}(\cdot))\ptc
$$
for some suitable continuous functions 
$\varphi^{i}:[0,T]\times \R^{i}\times \R\times \R^{n_{i}}\longrightarrow \R$ and $V^{j}:C([-1,0])\longrightarrow \R$, for any $1\leq j\leq n_{i}$.
\item Other classes of strategies fulfilling Assumption
\ref{assA} can be derived through infinite dimensional PDEs, see \cite{DGR} and \cite{DGRnote}.
\end{enumerate}
\end{rema}

\begin{teor} 	\label{teoNOAAR}
Let $s_0 > 0$, $\sigma : [0,1] \times C([-1,0]) \rightarrow \R  $ and two constants $c_1,c_2 >0$ such that $c_1\leq \sigma \leq c_2$. 
Suppose the following. 
\begin{enumerate}
\item The SDE $Y_{t}=s_{0} +\int_{0}^{t}\sigma(s,Y_{s}(\cdot))Y_{s}dW_s$ admits weak strictly positive existence for some Brownian motion $W$.
\item Let $S$ be such that $S_0 = s_0$ and $[S]_t = \int_0^t \sigma^2(s,S_s(\cdot)) S_s^2 ds $ (under the given probability $P$). 
\item $S$ fulfills the full support condition with respect to $P$.
\item Let $\sha$ be a class of self-financed portfolios $h$
 verifying Assumption \ref{assA}.
\end{enumerate}
Then the corresponding market is $\sha$-arbitrage free.
\end{teor}
\begin{rema} \label{RteoNOAAR}
Of course item 1. may be replaced with the
weak existence of the SDE
$  R_{t}= \log s_{0} + \int_{0}^{t}\sigma(s,R_{s}(\cdot)) dW_s$
for some real process $R$.
\end{rema}

\begin{proof} \
Let $h \in \sha$ be a self-financing portfolio and $\phi =(h_0,h)$
according to Proposition \ref{l5.5};
 let $X_t(\phi)$ be the wealth 
process such that $X_{0}(0)=0$ $P$-a.s. Without restriction of generality we can suppose that 
$X_{1}(\phi)=\int_{0}^{1} h_{s}d^-S_s$.
In reference to Assumption \ref{assA}, which is verified, we consider the corresponding continuous functional $\shv_\phi: C([0,1]) \rightarrow \R$. 
In particular $\mathcal{V} (S) = X_1(\phi)$ $P$-a.s. 
We suppose $X_1(\phi) \ge 0$ $P$-a.s. It remains
to show that $X_1(\phi) = 0$ $P$-a.s. 
We denote $\shc:=  C_{s_0}([0,1])$.
We first show that $\shv_\phi (\eta) \geq 0$
for any $\eta \in \shc$.
For this we suppose ab absurdo that it were not the case. Then there
would exist $\eta_0 \in \shc$ and $\varepsilon > 0 $
such that 
$ \shv_\phi (\eta) < 0 $ for all $\eta$ such that $\Vert \eta -
 \eta_0\Vert_\infty
\le \varepsilon$. 
Consequently
$$ P\{ X_1(\phi) < 0 \} = P\{\shv(S) < 0 \}\geq   P\{\shv_\phi(S) < 0; \Vert S - \eta_0 
 \Vert_\infty \le \varepsilon \}
> 0.$$
This contradicts the fact that $ X_1(\phi) \ge 0$ $P$-a.s.
It remains to prove that $X_{1}(\phi)=0$ $P$-a.s.\\
By assumption, let $\bar P$ a probability under which $S$ is a local martingale with 
$[S]_{t}=\int_{0}^{t}\sigma^{2}(s,S_{s}(\cdot))S_{s}^2 \, ds$. 
By assumption \ref{assA}, $\shv _{\phi}(S)=\int_{0}^{1}h_{s}dS_s$ $\bar P$-a.s. by Proposition \ref{p3.4}.
Consequently  $X_{1}(\phi)\geq 0$ $\bar P$-a.s. 
By Proposition \ref{p5.28} it follows that $h$ cannot be an arbitrage
 under the probability $\bar P$, if we show that $S$ is a $\bar P$-$\mathcal{A}$-martingale.
This is true whenever 
\[
\mathbb{E}_{\bar P}\left[ \int_{0}^1 \mathcal{H}^2(s,S_s(\cdot)) \sigma^2(s,S_s(\cdot))S_s^{2}ds\right]<\infty
\]
for every  $h_s= \mathcal{H}(s,S_s(\cdot)) $, as 
in Assumption \ref{assA}. 
This can be shown using the fact that $\mathcal{H}$ 
has polynomial growth and $\sigma$ is bounded.
In fact 
$
\mathbb{E}_{\bar P}\left[ \sup_{t\leq 1} |S_t|^q\right] < \infty
$
for every $q>1$ again
 using Burkholder-Davis-Gundy inequality and some exponential estimates.

Under $\bar P$ we have  $S_{t}=s_0 \, e^{M_{t}+A_{t}  }$ where $M_{t}= \int_{0}^t \sigma(s,S_s(\cdot))dW_s$ 
and $A_{t}=-\frac{1}{2}\int_{0}^{t} \sigma^{2}(s,S_s(\cdot))ds$. By item 1. of Proposition \ref{PFSC} $M$ fulfills the full support condition with respect to $\bar P$. 
By a Girsanov type argument, $M+A$ has the same property. 
By item 4. of Proposition \ref{PFSC} finally also $S$ fulfills the same condition (under $\bar P$).
By a similar reasoning as in the first part of the proof, we obtain that $\shv _\phi$ vanishes identically. 
Finally $X_{1}(\phi)=\shv_{\phi}(S)=0$ $P$-a.s. and this concludes the proof.

%

\end{proof}
\begin{rema} \label{Rfina}
\begin{enumerate}
\item Instead of applying Proposition \ref{p5.28} we could have used the classical theory of non-arbitrage, see \cite{DSbook} Theorem 14.1.1.
Their notion of non-arbitrage is however a bit different from ours. In that case one should restrict the class $\mathcal{A}$ requiring that the wealth process 
associated with $h$ is lower bounded by a (presumably negative) constant in order to avoid doubling strategies.
\item
An interesting question which is beyond the scope of our paper
is the following. Suppose that the underlying $S$ fulfills
the full support condition and that 
Assumption \ref{assA} is in force.
Is there any $\sha$-martingale measure?
\end{enumerate}
\end{rema}

\section{Utility maximization}

\subsection{An example of $\cal{A}$-martingale and  a related optimization problem} \label{3.26}
We illustrate a setting where Proposition \ref{c3.18} applies
and it provides a very similar results to
theorem 3.2 of \cite{LNN}. There, the authors study a particular case
of the optimization problem considered in Proposition \ref{c3.18}. 
As process $\xi$ they take a Brownian motion $W,$ and
they find sufficient conditions in order to have existence of a
process $\gamma$ such that $W-\int_0^\cdot \gamma_tdt$ is 
(in our terminology) an $\cal{A}$-martingale, being $\cal{A}$ some specific set we shall clarify later. To get their goal, they consider an anticipating
setting and combine Malliavin calculus with substitution formulae, the anticipation being generated by a random variable possibly depending on the whole trajectory of $W.$

We work into the specific framework of subsection \ref{5.25} . 

\begin{assu}\label{a3.21} We suppose the existence of a random variable $G$ in
$\mathbb{D}^{1,2},$ satisfying the following assumption:
\begin{enumerate}
\item $
\int_{\mathbb{R}}\mathbb{E}\pqa \vaa G\vac^2 I_{\pga
  0\leq x\leq G\pgc \cup \pga 0\geq x \geq G \pgc} \pqc dx<+\infty;
$
\item for a.a. $t$ in $[0,1]$ the process
$$
I(\cdot,t,G):=I_{[t,1]}(\cdot)I_{\pga \int_t^1 (D_s G)^2ds>0\pgc}\pta
\int_t^1 (D_s G)^2 ds \ptc^{-1}(D_t G)(D_{\cdot}G)
$$
belongs to $Dom\delta$ and there exists a $\mathcal{P}(\mathbb{F}) \times\mathcal{B}(\mathbb{R})$-measurable random field $\pta h(t,x), \interval, x\in \mathbb{R}\ptc$ such that $h(\cdot,G)$ belongs to $L^2\pta \Omega \times [0,1]\ptc$ and
$$
\mathbb{E}\pqa \int_0^1 I(u,t,G)dW_u \left|\right.\mathcal{F}_t \vee  \sigma(G)\pqc=h(t,G), \quad \interval.
$$
\end{enumerate}
\end{assu}
Let $\Theta(G)$ be the set of processes $\pta \theta_t, \intervala \ptc$ such that there exists a random field $(u(t,x), \interval, x\in \mathbb{R})$ with $\theta_t=u(t,G),$ $\intervala$ and
$$
\left\{
\begin{array}{ll}
u(t,\cdot)\in C^1(\mathbb{R}) \ \forall \ \interval. \\ \\
\int_{-n}^{n}\int_0^1 (\partial_xu(t,x))^2 dt dx <+\infty, \forall n \in \mathbb{N}  \ a.s.. \\ \\
\mathbb{E}\pqa \int_{\mathbb{R}}\pta  \int_0^1 (\partial_x u(t,x))^2 dt \ptc^2 dx+  \int_0^1 (u(t,0))^2 dt\pqc
 <+\infty. \\ \\
\mathbb{E}\pqa \int_0^1 (\partial_xu(t,G))^2 (D_tG)^2 dt
+\pta \int_0^1 (\partial_xu(t,G))^2dt \ptc \pta \int_0^1
(D_tG)^2 dt \ptc \pqc <+\infty.
\end{array}
\right.
$$
Suppose that $\cal{A}$ equals $\Theta(G).$ 
With the specifications above we have the following.
\begin{coro}\label{c4.31}
Let $b$ be a process in $L^2(\Omega \times [0,1]),$ such that $h(\cdot,G)+b$ belongs to the closure of $\cal A$ in $L^2(\Omega \times [0,1]).$ There exists an optimal process $\pi$ in $\cal{A}$ for the function 
$$
\theta \mapsto \mathbb{E}\pqa \int_0^1 \theta_td^-\pta W_t+\int_0^t b_sds\ptc-\frac{1}{2}\int_0^1 \theta_t^2dt \pqc
$$ if and only if $h(\cdot,G)+b$ belongs to $\cal{A}$ and $h(\cdot,G)+b=\pi$. 
\end{coro}
\begin{proof}
It is clear that $\cal{A}$ is a real linear space of measurable and with bounded paths processes verifying condition 1. of Assumption \ref{a0}.  
Proposition 2.8 of \cite{LNN} shows that every $\theta$ in $\cal{A}$ is in $L^2(\Omega \times [0,1]),$ that $\theta$ is  $W$-improperly forward integrable and that the improper integral belongs to $L^2(\Omega).$ In particular,  condition 2. of Assumption \ref{a0} is verified. Furthermore,  the proof of theorem 3.2 of \cite{LNN} implicitly  shows that the process  
$
W-\int_0^\cdot h(t,L)dt,
$
is a $\mathcal{A}$-martingale. This implies that $W+\int_0^\cdot   b_t dt -\int_0^\cdot \gamma_t dt, $ with $\gamma=h(\cdot,G)+b,$ is an $\mathcal{A}$-martingale. The end of the proof follows then by Proposition \ref{c3.18}
 setting $\xi = W+\int_0^\cdot   b_t dt.$ 
\end{proof}


\subsection{Formulation of the problem}

We consider the problem of maximization of expected utility from terminal wealth starting
from initial capital $X_0> 0,$ being $X_0$ a $\mathcal{G}_0$-measurable random variable. We define the function $U(x)$ modeling the
utility of an agent with wealth $x$ at the end of the trading period. The function $U$ is supposed to be of class $C^2((0,+\infty)),$ strictly increasing, with $U'(x)x$ bounded.

We will need the following assumption.

\begin{assu} \label{a5.41}
The utility function $U$ verifies  
$
\frac{U^{''}(x)x}{U'(x)}\leq -1, \quad \forall x>0.
$ 
\end{assu}

A typical example of function $U$ verifying Assumption \ref{a5.41} is $U(x)=\log(x).$

We will focus on portfolios with strictly positive value. As a consequence of this, before starting analyzing the problem of maximization, we show how it is possible to construct portfolio strategies when only positive wealth is allowed.

\begin{defi} \label{d5.7} For simplicity of calculation we introduce the process 
$$
A=\log(S)-\log(S_0)+\frac{1}{2}\int_0^\cdot\frac{1}{S^2_t}d\pqa S\pqc_t.
$$
\end{defi}

\begin{lemm} \label{l4.3}
Let $\theta=\pta \theta_t, \intervala \ptc$ be a $\mathbb{G}$-adapted process in $C^-_{b}([0,1))$ such that 
\begin{enumerate}
\item $\theta$ is $A$-improperly forward integrable. 
\item The process $A^\theta=\int_0^\cdot \theta_sd^-A_s$ has finite quadratic variation.
\item If $X^\theta$ is the process defined by  
$$
X^\theta=X_0\exp \pta \int_0^\cdot \theta_t d^-A_t+\int_0^\cdot  \pta 1-\theta_t\ptc dV_t-\frac{1}{2}\pqa A^\theta \pqc \ptc,
$$
then
$\int_0^\cdot X^\theta_t \theta_td^-A_t$ and $\int_0^\cdot X^\theta_t d^-\int_0^t \theta_s d^-A_s
$ improperly exist  and 
\beqn \label{c10} 
 \int_0^\cdot X^\theta_t d^-\int_0^t \theta_s d^-A_s =
  \int_0^\cdot X^\theta_t \theta_td^-A_t
\eeqn
\end{enumerate}
Then the couple $\pta X_0, h\ptc$, with  $h_t=\frac{\theta_t X^\theta_t}{S_t},$ $\intervala,$
is a self-financing portfolio with strictly positive wealth $X^\theta.$ In particular,  $\lim_{t\rightarrow 1} X^\theta_t=X^\theta_1$ exists and it is strictly positive.  
\end{lemm}

\begin{proof} Again, for simplicity we suppose $\tilde S = S$
therefore $V =0$.
Thanks to Proposition \ref{l2.1} $h$ is locally $S$-forward integrable and 
$
\int_0^\cdot h_t d^-S_t=\int_0^\cdot \theta_t X^\theta_td^-A_t.
$
Applying Corollary \ref{c2.15}, Proposition \ref{p2.1}, and using
hypothesis  3., $X ^\theta = \widetilde{X^\theta}$ 
can be rewritten in the following way:
\beqn  \label{e22}
X^\theta_t = X_0+\int_0^t \theta_s d^-A_s =
X_0+\int_0^t  h_sd^-S_s.
\eeqn
Proposition \ref{l5.5} tells us that $X^\theta$ is the wealth of the self-financing
 portfolio  $ \pta X_0, h\ptc.$
\end{proof}

\begin{rema}
The process $\theta$ in previous lemma represents the \textit{\textbf{proportion}} of wealth invested in $S.$
\end{rema}
\begin{rema}
Let $\theta$ be as in Lemma \ref{l4.3}. 
Then, for every $0\leq t < 1,$ $X$ is, indeed, the unique solution, on $[0,t],$ of equation
\beqn \nonu
X^\theta=X_0+\int_0^\cdot X^\theta_td^-\pta \int_0^t \theta_sd^-A_s +\int_0^t (1-\theta_s)dV_s-\frac{1}{2}\pqa A^\theta \pqc_t\ptc.
\eeqn
In fact, uniqueness is insured by Corollary 5.5 of \cite{RV2}.
It is important to highlight that, without the assumption on $\theta$ regarding the  chain rule in equality (\ref{c10}),  we cannot conclude that $X^\theta$ solves equation (\ref{e22}).
However we need to require that $X^\theta$ solves the latter equation to interpret it  as  the value of a portfolio whose proportion invested in $S$ is constituted by $\theta.$ 
In the sequel we will construct, in some specific settings, classes of processes defining proportions of wealth as in Lemma \ref{l4.3}. We will consider, in particular, two cases already contemplated in \cite{BO} and \cite{LNN}. Our definitions of those sets will result more complicated than the ones defined in the above cited papers. This happens because, in those works, the chain rule problem arising when the forward integral replaces the classical It\^o integral is not clarified.
\end{rema}

\begin{assu} \label{a5.40}
We assume the existence of a real linear space $\cal{A}^+$ of $\mathbb{G}$-adapted processes $(\theta_t,\intervala)$ in $C_b^-([0,1)),$  such that 
\begin{enumerate}
\item $\theta$ verifies condition 1., 2. and 3. of Lemma \ref{l4.3}, and $\pqa A^\theta\pqc=\int_0^\cdot \theta_t^2d\pqa A\pqc_t.$
\item $\theta I_{[0,t]}$ belongs to $\mathcal{A}^+$ for every $\intervala.$ 
\end{enumerate}
\end{assu}

For every $\theta$ in $\mathcal{A}^+$ we denote with $Q^\theta$ the probability measure  defined by:    $$\frac{dQ^\theta}{dP}=\frac{U'(X^\theta_1)X^\theta_1}{\mathbb{E}\pqa U'(X^\theta_1)X^\theta_1\pqc}.$$
 
 The utility maximization problem consists in finding a process $\pi$ in $\mathcal{A}^+$ maximizing the expected utility from terminal wealth, i.e.:
\beqn \label{21}
\pi=\arg \max_{\theta \in \mathcal{A}^+}\mathbb{E}\pqa U(X^\theta_1)\pqc.
\eeqn

Problem (\ref{21}) is not trivial because of the uncertain nature of the processes $A$ and $V$ and the non zero quadratic variation of $A.$ Indeed, let us suppose that $\pqa A\pqc=0$ and that both $A$ and $V$ are deterministic. Then,  
 it is sufficient to consider  $$\sup_{\lambda \in \mathbb{R}} \mathbb{E}\pqa U(X^\lambda_1)\pqc=\lim_{x\rightarrow +\infty}U(x),$$ and remind that $U$ is strictly increasing,  to see that a maximum can not be realized. 
The problem is less clear when the  term $-\frac{1}{2}\int_0^\cdot \theta_t^2d\pqa A\pqc_t$  and a source of randomness are added.

In the sequel, we will always assume the following.
\begin{assu}\label{5.42}
For every $\theta$ in $\cal{A}^+,$ 
$$
\mathbb{E}\pqa \vaa \int_0^1 \theta_td^- (A_t-V_t)\vac +\frac{1}{2}\int_0^1 \theta_t^2 \pqa A\pqc_t  \pqc<+\infty.
$$
\end{assu}

\begin{defi}
A process $\pi$ is said \textit{\textbf{optimal portfolio}} in $\cal{A}^+,$ if it is optimal for the function $\theta \mapsto \mathbb{E}\pqa U(X^\theta_1)\pqc$ in $\mathcal{A}^+,$ according to definition \ref{d3.18}.
\end{defi}

\begin{rema} Set $\xi=A-V,$ $\mathcal{A}=\mathcal{A}^+,$ and  
$$
F(\omega,x)=U\pta X_0(\omega)e^{ x+V_1(\omega)}\ptc, \quad (\omega,x) \in \Omega \times \mathbb{R}.
$$
According to definitions of section  \ref{s3.3.2}, $\mathcal{A}$ satisfies Assumption \ref{a0}, the function $F$ is measurable, almost surely in $C^1(\mathbb{R}),$ strictly increasing and with bounded  first derivative. If $U$ satisfies Assumption \ref{a5.41} then $F$ is also concave. Moreover $F(L^\theta)=U(X^\theta_1)$ for every $\theta$ in $\mathcal{A}^+.$  
\end{rema}

\subsection{About some admissible strategies}

Before stating some results about the existence of an optimal portfolio, 
we provide examples of sets of admissible strategies with positive wealth.

Similarly to section \ref{ss4.2}, it is possible to exhibit classes
of admissible strategies fulfilling the corresponding technical
assumption.
In the context of utility maximization that assumption is Assumption
\ref{a5.40}.

We omit  technical details since similar calculations were performed in
previous sections. We only supply precise statements.

\begin{enumerate} 
\item {\bf Admissible strategies via It\^o fields.}
For this example the reader should keep in mind subsection \ref{s5.1.1}.
 
\begin{prop}
Let $\cal{A}^+$ be the set of all processes $(\theta_t, \intervala)$
such that $\theta$ is the restriction to $[0,1)$ of a process $h$
belonging to
 $\mathcal{C}^1_A(\mathbb{G}).$ Then $\mathcal{A}^+$ satisfies the hypotheses of Assumption \ref{a5.40}.
\end{prop}
\item{\bf Admissible strategies via Malliavin calculus.} \label{s5.4.2}
We restrict ourselves to the setting of section \ref{5.25}. We recall that in that case $A=\int_0^\cdot \sigma_tdW_t+\int_0^\cdot \mu_t dt.$ We make the following additional assumption:

$$S^0=e^{\int_0^\cdot r_tdt},
$$ 
with $r$ in $L^{1,z}$ for some $z>4$ and $\mathbb{F}$-adapted.
 
\begin{prop} \label{p5.29}
Let $\cal{A}^+$ be the set of all $\mathbb{G}$-adapted processes in  $C^-_b([0,1))$ being the restriction on $[0,1)$ of processes $h$ in $L^{1,2}_{-}\cap L^{2,2},$  such that $D^-h$ is in $L^{1,2}_{-},$ and the random variable  
$$
\sup_{t\in [0,1]} \pta \vaa h_t \vac+\sup_{s\in[0,1]} \vaa D_s h_t\vac+\sup_{s,u\in[0,1]} \vaa D_s D_u h_t\vac \ptc 
$$
is bounded. 

Then $\cal{A}^+$ satisfies the hypotheses of Assumption \ref{a5.40}. 
\end{prop}
\item{\bf Admissible strategies via substitution.}

We return here to the framework of subsection \ref{e5.21}. 
\begin{prop}
Let $\mathcal{A}^+$ be the set of all processes which are the restriction to $[0,1)$ of processes in $\mathcal{A}^{p,\gamma}(G)$ for some $p>3$ and $\gamma>3d.$ Then $\mathcal{A}^+$ satisfies the hypotheses of Assumption \ref{a5.40}.  
\end{prop}
\end{enumerate}

\subsection{Optimal portfolios and $\mathcal{A}^+$-martingale property}
Adapting results contained in section \ref{s3.3.2} to the utility maximization problem, we can formulate the following propositions. We omit their proofs, being particular cases of the ones contained in that section.
  
\begin{prop}\label{p5.45}
If a process $\pi$ in $\mathcal{A}^+$ is an optimal portfolio, then  the process 
$
A-V-\int_0^\cdot \pi_td\pqa A\pqc_t
$
is an $\mathcal{A}^+$-martingale under ${Q^\pi}.$  If $U$ fulfills  Assumption \ref{a5.41}, then the converse holds.
\end{prop}

\begin{prop} \label{p5.46} 
Suppose that $\mathcal{A}^+$ satisfies Assumption $\mathcal{D}$ 
(see Definition \ref{d3.6}) with respect to $\mathbb{G}$. 
If a process $\pi$ in $\cal{A}^+$ is an optimal portfolio, then the process 
$A-V-\int_0^\cdot \beta_t d\pqa A \pqc_t
$
is a $\mathbb{G}$-martingale under $P,$ with   $$\beta=\pi+\frac{1}{p^\pi}
\frac{d\pqa p^\pi,A\pqc}{d\pqa A \pqc},\quad \mbox{ and   }
\quad p^\pi=\mathbb{E}^{Q^\pi}\pqa \frac{dP}{dQ^\pi}\left|\right. \mathcal{G}_\cdot\pqc.$$
If $U$ fulfills  Assumption \ref{a5.41}, then the converse holds.
\end{prop}

\begin{rema}
\begin{enumerate}
\item
We emphasize that if $U(x)=\log(x),$ then
the probability measure $Q^\pi$ appearing in Propositions \ref{p5.45} and \ref{p5.46} is equal to $P.$
\item
 In \cite{AI} it is proved that if the maximum of expected logarithmic utility over all  \textit{simple admissible strategies} is finite, then $S$ is a semimartingale with respect $\mathbb{G}.$ This result does not imply Proposition \ref{p5.46}. Indeed, we do not need to assume that our set of portfolio strategies contains the set of simple predictable admissible ones. On the contrary, we want to point out that, as soon as the class of admissible strategies is not \textit{large} enough, the semimartingale property of price processes could fail, even under finite expected utility.   
 \end{enumerate}
\end{rema}

\begin{prop}\label{p5.48}
Suppose that $U(x)=\log(x),$ $x$ in $(0,+\infty).$ Assume that there 
exists a measurable process $\gamma$ such that $A-V-\int_0^\cdot \gamma_t d\pqa A\pqc_t$ is an $\mathcal{A}^+$-martingale.
\begin{enumerate}
\item  If $\gamma$ belongs to $\mathcal{A}^+$ then it is an optimal portfolio.
\item Suppose moreover that
 there exists a sequence $(\theta^n)_{n\in \mathbb{N}}\subset \mathcal{A}^+$ 
such that
$$
\lim_{n \rightarrow +\infty}\mathbb{E}\pqa \int_0^1 \vaa \theta^n_t-\gamma_t\vac^2 d\pqa A\pqc_t \pqc=0.
$$ 
 If an optimal portfolio $\pi$ exists, then 
$d \pqa A\pqc \pga t\in [0,1),\pi_t \neq \gamma_t \pgc=0$  almost surely. 
\end{enumerate} 
\end{prop}

\subsection{Example}
We adopt the setting of section $\ref{s5.4.2}$ and we further assume that  $\sigma$ is a strictly positive real. 

\begin{prop}
If a process $\pi$ is an optimal portfolio in $\cal{A}^+,$ then the process 
$
W-\int_0^\cdot \pta \frac{r_t-\mu_t}{\sigma}+\pi_t\sigma\ptc  dt
$
is an $\cal{A}^+$-martingale under $Q^\pi.$ If $U$ fulfills Assumption \ref{a5.41}, then the converse holds.
\end{prop}
\begin{proof}
First of all we observe that it is not difficult to prove that $\mathcal{A}^+$ satisfies Assumption \ref{5.42}.
If a process $\pi$ is an optimal portfolio in $\cal{A}^+$ then Proposition \ref{p5.45} implies that the process $M^\pi$, with $M^\pi=\sigma\pta W-\int_0^\cdot \pta \frac{r_t-\mu_t}{\sigma}-\pi_t\sigma \ptc dt\ptc,$ is an $\cal{A}^+$-martingale under $Q^\pi.$
We observe that $\sigma^{-1}\mathcal{A}^+=\mathcal{A}^+.$ Therefore, $\sigma^{-1}M^\pi=W-\int_0^\cdot \pta \frac{r_t-\mu_t}{\sigma}+\pi_t\sigma\ptc  dt $ is an $\mathcal{A}^+$-martingale.

Similarly, if $U$ satisfies Assumption \ref{a5.41}, the converse follows by Proposition \ref{p5.45}. 
\end{proof}

\begin{coro}
Let $\mathcal{A}^+$ satisfy Assumption $\mathcal{D}$ with respect
 to $\mathbb{G}.$ 
If a process $\pi$ in $\cal{A}^+$ is an optimal portfolio then the process 
$
B=W-\int_0^\cdot \alpha_t  dt
$ with 
$$
\alpha=\pi\sigma+\frac{r-\mu}{\sigma}+\frac{1}{p^\pi}\frac{d\pqa p^\pi,W\pqc}{d\pqa W\pqc},\quad  \mbox{ and   }\quad p^\pi=\mathbb{E}^{Q^\pi}\pqa \frac{dP}{dQ^\pi}\left|\right. \mathcal{G}_\cdot\pqc,
$$
is a $\mathbb{G}$-Brownian motion under $P.$
If  $U$ satisfies Assumption \ref{a5.41}, then the converse holds.
\end{coro}
\begin{proof}
Let $\pi$ be an optimal portfolio. By Proposition \ref{p5.1}, the
process $B$
 is a $\mathbb{G}$-martingale and so a $\mathbb{G}$-Brownian motion under $P.$
\end{proof}
The results concerning the example above were proved in \cite{BO}. We generalize them in two directions: we replace the geometric Brownian motion $A$ by a finite quadratic variation process and we let the set of possible strategies vary in sets which can, a priori, exclude some simple predictable processes.

\subsection{Example}
We consider the example treated in section \ref{3.26}. We suppose, for simplicity, that
$$
S_t=S_0 e^{ \sigma W_t+ \pta \mu-\frac{\sigma^2}{2}\ptc t}, \quad S^0_t=e^{rt}
\quad \interval,
$$
being $\sigma,$ $\mu$ and $r$ positive constants.  This implies $A_t=\sigma W_t+\mu t,$ and $V_t=rt$ for $\interval.$ 
We set $\mathcal{A}^+=\Theta(L).$

\begin{prop}
Suppose that $U(x)=\log(x),$ $x$ in $(0,+\infty).$ Suppose that $h(\cdot,L)$  belongs to the closure of $\Theta(L)$ in $L^2(\Omega \times [0,1]).$ Then an optimal portfolio $\pi$ exists if and only if the process $h(\cdot,L)+\int_0^\cdot \frac{\mu-r}{\sigma}dt $ belongs to $\Theta(L)$ and $\pi=h(\cdot,L)+\frac{\mu-r}{\sigma}.$ 
\end{prop}

\begin{proof}
The result follows from Corollary \ref{c4.31}. 
\end{proof}

Sufficiency for the Proposition above was shown, with more general $\sigma,$ $r$ and $\mu$ in theorem 3.2 of \cite{LNN}. Nevertheless, in this paper we go further in the analysis of utility maximization problem. Indeed, besides observing that the converse of that theorem holds true, we find that the existence of an optimal strategy is strictly connected, even for different choices of the utility function, to the $\cal{A}^+$-semimartingale property of $W.$ To be more precise, in that paper the authors show that an optimal process exists, under the given hypotheses, handling directly the expression of the expected utility, which has, in the logarithmic case, a nice expression. Here we reinterpret their techniques at a higher level which permits us to partially generalize those results.

\addcontentsline{toc}{chapter}{Bibliography}
\nocite{*}
\bibliographystyle{plain}
\bibliography{insider}

\end{document}